\documentclass[12pt,a4paper]{article}
\usepackage{fullpage}
\usepackage[utf8]{inputenc}
\usepackage[T1]{fontenc}
\usepackage{amsmath}
\usepackage{amsfonts}
\usepackage{amssymb}
\usepackage{graphicx}
\usepackage{amsthm}
\usepackage[all,cmtip]{xy}
\usepackage{tikz-cd}

\usepackage[draft]{say}

\usepackage[bookmarks=true, bookmarksopen=true,%
bookmarksdepth=3,bookmarksopenlevel=2,%
colorlinks=true,%
linkcolor=blue,%
citecolor=blue,%
filecolor=blue,%
menucolor=blue,%
urlcolor=blue]{hyperref}

\numberwithin{equation}{section}
\newtheorem{thm}[equation]{Theorem}
\newtheorem{theorem}[equation]{Theorem}
\newtheorem{lem}[equation]{Lemma}
\newtheorem{lemma}[equation]{Lemma}
\newtheorem{cor}[equation]{Corollary}
\newtheorem{corollary}[equation]{Corollary}
\newtheorem{prop}[equation]{Proposition}
\newtheorem{proposition}[equation]{Proposition}
\newtheorem{thmintro}{Theorem}

\theoremstyle{definition}
\newtheorem{dfn}[equation]{Definition}

\theoremstyle{remark}
\newtheorem{rem}[equation]{Remark}
\newtheorem{remark}[equation]{Remark}
\newtheorem{ex}[equation]{Example}
\newtheorem{example}[equation]{Example}

\newtheorem{question}[equation]{Question}

\newcommand{\Sp}{\mathrm{Sp}}
\newcommand{\A}{\mathcal{A}}
\renewcommand{\P}{\mathcal{P}}
\newcommand{\E}{\mathcal{E}}

\newcommand{\Z}{\mathbf{Z}}

\newcommand{\CC}{\mathcal{C}}
\newcommand{\HH}{\mathcal{H}}
\newcommand{\R}{\mathbf{R}}
\newcommand{\F}{\mathcal{F}}
\newcommand{\LL}{\mathcal{L}}

\newcommand{\N}{\mathbf{N}}
\newcommand{\NN}{\mathcal{N}}
\newcommand{\DD}{\mathcal{D}}
\newcommand{\RR}{\mathcal{R}}

\newcommand{\Map}{\mathrm{Map}}
\newcommand{\del}{\partial}
\newcommand{\s}{\mathbf{s}}
\newcommand{\Sh}{\mathrm{Sh}}
\renewcommand{\S}{\mathbf{S}}

\newcommand{\id}{\mathrm{id}}

\newcommand{\G}{\mathcal{G}}
\newcommand{\II}{\mathcal{I}}
\DeclareMathOperator{\Hom}{Hom}

\DeclareMathOperator{\crit}{crit}

\DeclareMathOperator{\colim}{colim}

\DeclareMathOperator{\pr}{pr}
\renewcommand{\k}{\mathbf{k}}
\newcommand{\GF}{\mathrm{GF}}
\newcommand{\Pic}{\mathrm{Pic}}

\title{A topological classification of generating functions}
\author{Sylvain Courte and Vivek Shende}
\date{}

\usepackage{tocloft}
\begin{document}

\maketitle

\begin{abstract}
    From a generating function for a Legendrian in a $1$-jet bundle, we may extract the following topological information: (1) a trivialization of the
    stable Gauss map, (2) the sheaf of sub-level-set stable cohomotopies, and (3) an identification of the microlocalization of the latter with the J-homomorphism image of the former.  Here we show that in fact (1), (2), (3) completely classify generating functions up to the classical equivalence relations of stabilization and fiberwise diffeomorphism.
\end{abstract}

\small
\setcounter{tocdepth}{2}
\renewcommand\contentsname{\vspace*{-20pt}}
\tableofcontents
\normalsize
\thispagestyle{empty}

\section{Introduction}\label{sec:intro}

The use of generating functions to present Lagrangian submanifolds appeared already in the works of Hamilton on classical mechanics and of Jacobi on geometric optics. Closer to modern times, this method was profitably employed  by Arnol'd \cite{Arnold} and H\"ormander \cite{Hormander}, for purposes of singularity theory and analysis, respectively.

Let us recall the basic idea.  If $M$ is a manifold and $f: M \to \R$ is a function, then the graph of $df$ is a Lagrangian in $T^*M$.  More generally, if $\pi: E \to M$ is any map (for us, a trivial bundle with fiber $\R^n$ or an open set thereof), then the conormal to the graph of $\pi$ is a Lagrangian correspondence, which can be used to transform Lagrangians in $T^*E$ to Lagrangian immersions in $T^* M$.  In particular, given $f: E \to \R$, we may obtain in this way, from the graph of $df$, an immersed Lagrangian $L_f \subset T^*M$ (under a suitable transversality condition), parameterized by $\pi$-fiberwise critical points of $f$.  In fact, such a Lagrangian is necessarily exact, i.e. the image of a Legendrian in the 1-jet bundle $J^1 M$.  To see this, transform the conormal to the graph of $f$ in $T^*(E \times \R)$ by  the correspondence derived as above from $\pi\times \id:E\times \R \to M \times \R$; the result is a conical Lagrangian in $T^*(M \times \R)$.  Fixing the value of the cotangent vector in the $\R$ direction gives the desired Legendrian $\Lambda_f \subset J^1 M$.

There is a classical equivalence relation among generating functions, generated by pullback by $\pi$-fiberwise diffeomorphism, and fiberwise sum with a  quadratic form.  It will be most natural for us to restrict to constant (over $M$) quadratic forms of symmetric signature $(k,k)$.
We are interested in classifying generating function presentations of a given Legendrian $\Lambda$, up to this equivalence. In fact this will lead
us to investigate the homotopy type of a space, denoted $\GF_\Lambda(M)^*$, of all generating functions for $\Lambda$, the connected components
of which will correspond to equivalence classes.

\vspace{2mm}

There is a classification, due to Giroux \cite{Giroux} and Latour \cite{Latour},
of the germs of  equivalence classes of generating functions  along  their critical loci (we refer to such germs as {\em semi-local generating functions}).  Recall the stable Gauss map $\gamma: \Lambda \to U/O$, defined by comparing along $\Lambda$ the polarizations given by the tangent to $\Lambda$ and the cotangent fiber.
A semi-local generating function provides a null-homotopy of this map: the tangent to the graph of $df$ provides a lift of the stable Gauss map to a contractible space of graphical Lagrangians.   We will define a suitable space (simplicial set), denoted $\mu\GF_\Lambda(M)^*$,
of all semi-local generating functions for $\Lambda$ (see Definition~\ref{dfn:muGF}). The result of Giroux and Latour can be phrased as
a homotopy equivalence between this space and the (possibly empty) space of all null-homotopies of the stable Gauss map.
Note that this classification implies that semi-local generating functions cannot detect the subtle global phenomena of symplectic topology, e.g. nonsqueezing, the expected lower bounds on Lagrangian intersections, etc.

Global generating functions, now required to be defined on all of $E = M \times \R^n$ and to be appropriately ``tame at infinity, fiberwise for $\pi: E \to M$''  (a notion whose properties we develop in Appendix~\ref{app:fibinfty}), do detect such global phenomena, as shown in the  works of Chaperon \cite{Chaperon}, Laudenbach and Sikorav \cite{Laudenbach-Sikorav}, Sikorav \cite{Sikorav}, Chekanov \cite{Chekanov},  Eliashberg and Gromov \cite{Eliashberg-Gromov}, and Viterbo \cite{Viterbo}, among many others.  It follows that their classification must be richer than that of semi-local generating functions.  Yet, in the only examples where the complete classification was previously known -- the zero section of a jet bundle \cite{Viterbo, Theret-viterbo-uniqueness} and the 1-dimensional Legendrian unknot \cite{Jordan-Traynor} --  each (equivalence class of) generating function is determined by its underlying semi-local generating function.

\vspace{2mm}

We may obtain additional invariants of global generating functions from the topology of fiberwise sublevel sets of $f$.  Such invariants have been considered since \cite{Viterbo}; we record them using the formalism of sheaf theory.  As we consider generating functions equivalent if they may be related by stabilization, only the stable homotopy types of said sublevel sets are invariant; to capture these, we work with sheaves valued in the category $\Sp$ of spectra.

Let us briefly recall how to produce the sheaf associated to a generating function (see Section \ref{sec:GFtoSh} for more details).  Let $\S$ be the sphere spectrum, and let $\S_{f \le z}$ be the constant sheaf on the  region not below the graph of $f$ in $E \times \R$.  There is a sheaf
on $M \times \R$ whose stalk at $(m, z)$ is the stable cohomotopy spectrum of $\{ y \in \R^n \,|\, f(m, y) \le z\}$, shifted by $n/2$, namely:\footnote{Thus our sheaves may take value in the category of spectra formally shifted by $\frac{1}{2}$. The reader annoyed by this can  restrict their attention to generating functions where the fiber dimension $n$ of $E$ is even.}\footnote{The condition that $f$ is tame allows to perform a fiberwise compactification without changing homotopy  type, so that we can apply proper base change to compute the stalks.}
$$\s'(f) := (\pi\times \id)_* \S_{f \le z}[n/2].$$
Essentially for reasons of convenience and normalization, we prefer to view $\s'(f)$ as living
in Tamarkin's quotient $\Sh^+(M)$  of the category of sheaves on $M\times \R$ (see Section~\ref{sec:Sh}), or equivalently, to replace $\s'(f)$  with the sheaf $\s(f)$ giving the relative cohomotopies with respect to a sufficiently negative sublevel set (because we have assumed tameness, the sublevel set cohomotopy stabilizes as $z \ll 0$).

To any sheaf $F$ on a manifold $M$, there is associated the {\em microsupport} $ss(F) \subset T^*M$, consisting of the ``codirections along which sections fail to propagate'' \cite{Kashiwara-Schapira}.  This locus is always conic, is always co-isotropic \cite[Thm. 6.5.4]{Kashiwara-Schapira}, and is sometimes, only for the simplest sheaves, Lagrangian.
For a tame generating function $f$, the microsupport of $\s(f)$ is precisely $\Lambda_f$ (see Lemma~\ref{gen fn sheaf support}).

More generally, we regard a sheaf whose microsupport similarly characterizes a given Legendrian
(let us call them {\em generating sheaves})
as homological (or, here, stable homotopical) shadows of generating functions.  Classifying generating sheaves is an essentially combinatorial/topological question in terms of the front diagram (see e.g. \cite{STZ}).
Generating sheaves have many applications in symplectic topology, and in particular can be used in place of a generating function when one does not or is not known to exist \cite{Tamarkin, GKS, Guillermou, Shende-conormal, Chiu-nonsqueezing, Zhang-capacities, Asano-heavy}.
The sheaf methods are also of interest lately on account of the comparison theorems with the holomorphic curve methods \cite{Nadler-Zaslow, Nadler, GPS3, Viterbo-sheaf}, and consequent applications in homological mirror symmetry, such as \cite{FLTZ, Kuwagaki, Gammage-Shende, GMW}.

The classification problem for generating functions is thus reduced to two parts: classifying generating sheaves (which we will not discuss here, or said differently, we regard the sheaf itself as a topological invariant of the function) and determining when a generating sheaf arises from a generating function.
Let us mention some obvious obstructions.  One typically restricts attention to generating functions $f$ for which the fiberwise critical locus is transversely cut out and maps diffeomorphically onto $\Lambda$; the corresponding property of $\mathbf{s}(f)$ is that $\mathbf{s}(f)$ is a `simple' sheaf in the sense of \cite[Def. 7.5.4]{Kashiwara-Schapira}.  The notion of simple sheaf captures the idea that, as you cross the microsupport, the dimension of the space of sections changes by one.
In fact, this one-dimensional-change forms a sort of local system along the microsupport, the study of which will bring us again into contact with trivializations of the stable Gauss map.   We recall the formalism for describing this local system.

Given a manifold $N$ (we take $N = M \times \R$), there is a sheaf of categories defined on $T^*N$ by sheafifying the presheaf:
$$\mu \Sh^{pre}(U) := \Sh(N) / \{F\,|\, ss(F) \cap U = \emptyset \}$$
When $N = M \times \R$, we write also $\mu \Sh$ for the pullback of this sheaf of categories to $J^1 M$.  For $\Lambda \subset J^1 M$, we write $\mu \Sh_\Lambda$ for the subsheaf of full subcategories of objects supported on $\Lambda$.  It is the pushforward of a sheaf of categories on $\Lambda$.  The natural map $\Sh(M) \to \mu \Sh(J^1 M)$ carries a sheaf microsupported in $\Lambda$ to a global section of $\mu \Sh_\Lambda$.

If we are working with sheaves with coefficients in some symmetric monoidal category $\mathcal{C}$, there is a canonical group homomorphism \cite[Sec. 11]{Nadler-Shende}:
$$\mu: U/O \to B \Pic(\mathcal{C}).$$
Composing with the stable Gauss map $\Lambda \to U/O$ gives  $\mu_\Lambda: \Lambda \to B\Pic(\mathcal{C})$, or in other words, a $\Pic(\mathcal{C})$ local system over $\Lambda$, which we also denote $\mu_\Lambda$.
Then$$\mu \Sh_\Lambda = \mu_\Lambda \times^{\Pic(\mathcal{C})} \mathcal{C}$$
and as such, an invertible global section of $\mu \Sh_\Lambda$ provides a trivialization of $\mu_\Lambda$ and
$\mu \Sh_\Lambda$.  A sheaf $F$ on $M \times \R$ with $ss(F)|_{J^1 M} = \Lambda$ tautologically provides a section of $\mu \Sh_\Lambda$; invertibility of the section is equivalent to the condition that $F$ is simple in the sense of Kashiwara and Schapira.  In summary, a simple sheaf microsupported in $\Lambda$ provides a null-homotopy of $\mu_\Lambda$. In the case where $\CC=\Sp$, Jin further identified (see \cite{Jin}) the map $U/O\to B\Pic(\CC)$ as (a delooping of) the J-homomorphism
$U/O\simeq B(\Z\times BO)\to B(\Z\times BG)\simeq B\Pic(\CC)$ (recall $BO$ is the classifying space of stable vector bundles
while $BG$ is the classifying space of stable spherical fibrations and the $J$-homomorphism maps
a vector bundle to the spherical fibration obtained by one point-compactifying each fiber).

In case $F = \s(f)$, then the null-homotopy of $\mu_\Lambda$ furnished by $F$ is
the composition of $\mu$ with the null-homotopy of the stable Gauss map for $f$ (e.g. consider both via the symplectic reduction from the graph of $f$).
That is, we have a commutative diagram
\begin{equation}\label{eq:mainsquare}
\begin{tikzcd}
\GF_\Lambda(M)^* \arrow[r,"\s"] \arrow[d] & \Sh^+_\Lambda(M;\S)^* \arrow[d]\\
\mu\GF_\Lambda(M)^* \arrow[r,"\mu^*"] & \mu \Sh^+_\Lambda(M;\S)^*.
\end{tikzcd}
\end{equation}
Here, we write $\GF_{\Lambda}(M)^*$ for the space (simplicial set) of all tame generating functions for $\Lambda$ (we pass to colimit under stabilization by a fixed quadratic form, the ${}^*$ indicates we impose the usual transversality condition, see Definition \ref{dfn:spaceGF})
and $\mu\GF_\Lambda(M)^*$ for the analogous space of semi-local generating functions.
We write $\Sh^+_\Lambda(M;\S)^*$ for the space (simplicial set) of simple sheaves microsupported in $\Lambda$ and $\mu \Sh_\Lambda^+(M;\S)^*$ for the space
of microsheaves for $\Lambda$ (valued in spectra, i.e. $\S$-modules). The spaces $\mu\GF_\Lambda(M)^*$ and $\mu\Sh^+_\Lambda(M)^*$ are well-understood:
they are homotopy equivalent to the spaces of null-homotopies of the stable Gauss map $\Lambda \to U/O$ and of $\Lambda \to U/O \to B \Pic(\Sp)$ respectively
(see Theorem~\ref{thm:spacelevelgiroux-latour} and Theorem~\ref{thm:spaceleveljin}).

Our main result is the following.
\begin{thmintro} \label{thm:main}
 Let $M$ be a manifold and $\Lambda\subset J^1 M$ a closed Legendrian submanifold.
The diagram~\eqref{eq:mainsquare} is a homotopy pullback square.
\end{thmintro}

This means that we have a homotopy equivalence
\begin{equation}\label{eq:mainfiberedproduct}
\GF_\Lambda(M)^* \simeq \mu \GF_\Lambda(M)^* \times_{\mu \Sh^+_\Lambda(M;\S)^*} \Sh_\Lambda(M)^*
\end{equation}
which we can consider as a ``topological classification'' of generating functions.
The homotopy fibers of the horizontal maps in \eqref{eq:mainsquare} thus agree; so we see that lifts of a sheaf to a generating function
are classified by lifts of the induced null-homotopy of $\mu_\Lambda$ to a null-homotopy of the stable Gauss map of $\Lambda$.

There is a natural ring map from the sphere spectrum $\S$ to the integers $\Z$; tensoring over it carries sheaves over $\S$ to sheaves over $\Z$.
This leads to the  commutative square:
\begin{equation}\label{eq:StoZ}
\begin{tikzcd}
\Sh^+_\Lambda(M;\S)^* \arrow[r]\arrow[d]& \mu\Sh^+_{\Lambda}(M;\S)^*\arrow[d] \\
\Sh^+_\Lambda(M;\Z)^* \arrow[r] &  \mu\Sh^+_\Lambda(M;\Z)^*.
\end{tikzcd}
\end{equation}
In contrast with \eqref{eq:mainsquare}, the square \eqref{eq:StoZ} is not cartesian in general, i.e.
the map
\begin{equation}\label{eq:mainfiberedZtoS}
\Sh_\Lambda^+(M;\S)^* \to  \Sh_\Lambda^+(M;\Z)^* \times_{\mu \Sh^+_\Lambda(M;\Z)^*} \mu\Sh^+_\Lambda(M;\S)^*
\end{equation}
is not a homotopy equivalence, already when $\Lambda$ is the Hopf link (see Example~\ref{ex:hopflink}).

Under some assumptions on the gradings (i.e. Conley-Zehnder indices) of the Reeb chords, we can still say something
using the same method of proof as for Theorem \ref{thm:main}.
Say our convention on grading is so that the $1$-dimensional Legendrian unknot has a single Reeb chord of degree $1$.
For a knot, the grading of Reeb chords is then canonical, however for a link the Maslov potential
of each component can be shifted independently. The datum of a microlocal sheaf $F\in \mu \Sh^+_\Lambda(M;\Z)$
(and a fortiori of a genuine sheaf) gives a well-defined integer-valued grading for the Reeb chords.

\begin{thmintro} \label{thm: Z to S}
   Let $M$ be a manifold and $\Lambda\subset J^1 M$ a closed Legendrian submanifold.
The map \eqref{eq:mainfiberedZtoS} is $l$-connected when restricted over a component of $\mu\Sh^+_\Lambda(M;\Z)^*$
giving degree $\leq 1-l$ to all Reeb chords of $\Lambda$.
\end{thmintro}

We give examples and applications in Section \ref{sec: examples}.

\vspace{2mm}
{\bf Acknowledgements.}  VS is supported by Villum Fonden Villum Investigator grant 37814, Novo Nordisk Foundation grant NNF20OC0066298, and Danish National Research Foundation grant DNRF157.
SC would like to warmly thank Stéphane Guillermou for countless stimulating discussions about sheaves and generating functions, and also Antoine Fermé who contributed actively at an early stage of the project.
\section{Generating functions}\label{sec:GF}

\subsection{Tame generating functions}
We will say that a function $f: M \times \R^n \to \R$ is {\em fiberwise tame} if there is a closed set $K\subset M\times \R^n$ and a smooth function $a:M\to (0,+\infty)$
such that the projection $K\to M$ is proper and the maps
\[\{(x,v)\in M \times \R^n, |f(x,v)|\geq a(x)\} \to \{(x,z)\in M\times \R, |z|\geq a(x)\},\]
and
\[\{(x,v)\in M \times \R^n, |f(x,v)|\leq a(x)\} \setminus K \to \{(x,z)\in M\times \R, |z|\leq a(x)\}\]
are fibrations. Examples include:
\begin{itemize}
\item (Quadratic at infinity): $f(x,v)=Q(v)+\epsilon(x,v)$ where $Q$ is a non-degenerate quadratic form and $\epsilon$ has proper support over $M$,
\item (Linear at infinity): $f(x,v,w)=w+\epsilon(x,v,w)$ where $\epsilon$ has proper support over $M$.
\end{itemize}

We develop some elementary properties of fiberwise tameness in Appendix~\ref{app:fibinfty}.

\begin{dfn}[Generating functions]\label{dfn:gf}
A \emph{generating function over $M$} is a function
\[f:M\times \R^n\to \R\]
for some $n\in \N$ such that
\begin{itemize}
\item $f$ is \emph{fiberwise tame} with respect to the projection $M\times \R^n \to M$,
\item the map $M\times \R^n\to \R^n$ given by $(x,v) \mapsto \frac{\del f}{\del v}(x,v)$ is transverse to $0$.
\end{itemize}
The subset
\[\Sigma_f=\{(x,v)\in M\times \R^n, \frac{\del f}{\del v}(x,v)=0\}\]
is called the \emph{singular set} of $f$ (with respect to the projection $M\times \R^n \to M$),
it is a submanifold of codimension $n$.
If $M$ has boundary and corners (this will be the case when we will take product by a simplex to define the simplicial set $\GF_\Lambda(M)^*$), we require the same properties for the restriction of $f$ to each stratum $M_i\subset \del M$ of any codimension.
\end{dfn}

The following property is well-known.
\begin{prop}
The map $i_f:\Sigma_f \to J^1 M$ given by
\[i_f(x,v)=(x,\frac{\del f}{\del x}(x,v),f(x,v))\]
is a Legendrian immersion.
\end{prop}

\begin{dfn}[Excellent generating functions]
A generating function $f:M\times \R^n\to \R$ is \emph{excellent} if the map $i_f$ is an embedding. In this case we say that
$i_f(\Sigma_f)$ is generated by $f$, and sometimes denote $i_f(\Sigma_f)=\Lambda_f$.
If $M$ has boundary and corners, we require the same property for the restriction of $f$ to each stratum $M_i\subset M$ of any codimension.
\end{dfn}

It can be proven using Thom's transversality theorem that a generic fiberwise tame function
$M\times \R^n\to \R$ is an excellent generating function.

\begin{ex}
If $M$ is a single point, $f:\R^n\to \R$ is excellent if and only if it is a Morse function with distinct critical values.
If $M=[0,1]$, $(f_t)_{t\in[0,1]}:\R^n \to \R$ is excellent if and only if the restrictions to $\{0\}\times \R^n$, $\{1\}\times \R^n$ and  $(0,1)\times \R^n$ are excellent. This says in particular that the Cerf diagram of the path $(f_t)_{t\in [0,1]}$ (projection of $i_f(\Sigma_f)$ onto $J^0 M=M\times \R$)
is the front of a Legendrian submanifold transverse to $t=0$ and $t=1$.
\end{ex}

\begin{rem}
There is a more general notion of generating functions where we replace $M\times \R^n\to M$ by any fibration $p:E\to M$
and require $f$ to be fiberwise tame with respect to $p$. For the comparison with sheaves to be simpler, we stick in this paper
to the case where $p$ is a trivial bundle with $\R^n$-fibers.
\end{rem}

\begin{dfn}[Stabilization]
We fix once and for all a non-degenerate quadratic form $h$ on $\R^2$ of signature $0$ (e.g. $h(u,v)=uv$), and denote $h^k:\R^{2k}\to \R$ with $k \geq 0$
the direct sum of $k$ copies of $h$. Given a generating function $f:M\times \R^n$, the function
\[\sigma^k f : M\times \R^n\times \R^{2k} \to \R\]
defined by
\[\sigma^k f(x;u,v)=f(x;u)+h^k(v)\]
is called the \emph{$k$-fold stabilization} of $f$.
\end{dfn}

The $k$-fold stabilization of $f$ is also fiberwise tame (e.g. per Lemma~\ref{lem:tamesum}) and satisfies the transversality condition of Definition~\ref{dfn:gf},
hence it is a generating function over $M$. The singular set is
\[\Sigma_{\sigma^k f}= \Sigma_f \times \{0\}\]
and the generated Legendrian immersion is unchanged: for all $(x;u)\in \Sigma_f$,
\[i_{\sigma^k f}(x,u,0)=i_f(x,u).\]
Moreover, $f$ is excellent if and only if $\sigma^k f$ is excellent.

\begin{dfn}[Equivalence]
We say that two generating functions $f:M\times \R^n\to \R$ and $g:M\times \R^m\to \R$
are \emph{equivalent} if there exists a diffeomorphism $\psi : M\times \R^n\to M\times \R^m$
fibered over $M$ such that $g\circ \psi=f$ (note that this implies $n=m$).
\end{dfn}

\begin{dfn}[Stable equivalence]
We say that two generating functions $f:M\times \R^n\to \R$ and $g:M\times \R^m \to \R$ are \emph{stably equivalent} if there exists $k,l\in \N$ such that $\sigma^k f$
and $\sigma^l g$ are equivalent (note that this implies $n-m$ is even).
We sometimes write isomorphic to mean stably equivalent.
\end{dfn}

Stably equivalent excellent generating functions generate the same Legendrian submanifold.

The following celebrated result shows that generating function persist under Legendrian isotopy.

\begin{thm}[Chekanov, Chaperon-Théret]\label{gen fn homotopy lifting 1}
Let $f:M\times \R^n\to \R$ be an excellent generating function for a Legendrian submanifold $\Lambda$.
Let $(\Lambda_t)_{t\in [0,1]}$ be a compactly supported Legendrian isotopy starting at $\Lambda_0=\Lambda$.
Then there exists an excellent generating function $f:M\times[0,1]\times \R^n\times \R^{2k}\to \R$ over $M\times [0,1]$
such that $f_0=\sigma^k f$ and for each $t\in[0,1]$, the restriction $f_t$ to $M\times \{t\}\times \R^n\times \R^{2k}$ is an excellent generating function for $\Lambda_t$.
\end{thm}

\subsection{Semi-local generating functions}

\begin{dfn}[semi-local generating function, version 1]\label{dfn:semilocalgf1}
A \emph{semi-local generating function} over $M$ is an open set $U \subset M\times \R^n$
and a function $f:U\to \R$ (not necessarily tame) such that $\frac{\del f}{\del v}:U\to \R^n$
is transverse to $0$ and $\Sigma_f=(\frac{\del f}{\del v})^{-1}(0)\to M$ is proper.
It is excellent if $i_f:\Sigma_f \to J^1 M$ is an embedding.
\end{dfn}

The germ of $f$ near $\Sigma_f$ can always be extended to a fiberwise tame function $\tilde{f}:M\times \R^n\to \R$
in order to obtain a generating function such that $\Lambda_f=\Lambda_{\tilde{f}}$ near $\Lambda_f$. Indeed
we may take a generic function $\tilde{f}$ which satisfies $\tilde f(x,v)=\|v\|^2$ at infinity and
$\tilde{f}(x,v)=f(x,v)$ near $\Sigma_f$.
This suggests the following variant, which we take as our official definition.

\begin{dfn} [semi-local generating function, version 2]\label{semilocal gf def}
A \emph{semi-local generating function} over $M$ is a couple $(f,\Sigma)$ where $f:M\times \R^n\to \R$
is a generating function (hence fiberwise tame) and $\Sigma \subset \Sigma_f$ is a closed and open subset.
It is excellent if $i_f:\Sigma \to J^1 M$ is an embedding.
\end{dfn}

Given such $(f,\Sigma)$, there is an open set $U\subset M\times \R^n$ such that $\Sigma_f\cap U=\Sigma$
and the restriction of $f$ to $U$ is a semi-local generating function in the sense of Definition~\ref{dfn:semilocalgf1}.

\begin{dfn}\label{dfn:semilocalequiv}
Two semi-local generating functions $(f,\Sigma),(g,\Sigma')$ are \emph{semi-locally equivalent} if there exists
a fibered diffeomorphism $\psi:M\times \R^n\to M\times \R^n$ such that $\psi(\Sigma)=\Sigma'$
and $g\circ \psi=f$ near $\Sigma$.
\end{dfn}

In particular, this equivalence relation identifies semi-local generating functions with their germs
along their singular set and the two versions above agree at the level of equivalence classes.

\begin{dfn}\label{dfn:stable_sl_equiv}
Two semi-local generating functions $(f,\Sigma)$ and $(g,\Sigma')$ over $M$ are \emph{stably semi-locally equivalent}
if $(\sigma^k f,\Sigma)$ is semi-locally equivalent to $(\sigma^l g,\Sigma')$ for some $k,l \in \N$. We sometimes write
isomorphic to mean semi-locally stably equivalent.
\end{dfn}

In fact, if the fibered diffeomorphism $\psi$ in Definition~\ref{dfn:semilocalequiv} is only defined between neighborhoods
$V$ and $V'$ of $\Sigma$ and $\Sigma'$ respectively (that would be the natural equivalence relation following Definition~\ref{dfn:semilocalgf1}),
then, after stabilization, it can be extended to a global fibered diffeomorphism of $M\times \R^n$.
Indeed the restriction of $\psi$ to $\Sigma$ is fibered isotopic to the inclusion provided $n$ is large enough
and moreover we can arrange the framing by further stabilization so that $\psi$ is fibered isotopic
to the identity map in a tubular neighborhood of $\Sigma$.

Of course a genuine generating function $f$ can be considered as a semi-local generating function $(f,\Sigma)$ with $\Sigma=\Sigma_f$
and stably equivalent generating functions induce semi-locally stably equivalent semi-local generating functions.

Recall that any compact Legendrian submanifold $\Lambda \subset J^1 M$ (or more generally a compact Lagrangian immersion
$L\to T^*M$) has a well-defined (up to homotopy)
stable Lagrangian Gauss map
\[\gamma_\Lambda : \Lambda \to U/O\]
which records how the tangent spaces $T_x \Lambda$ (stably) turn with respect to the vertical distribution (fibers of $J^1 M \to J^0 M$).
Giroux and Latour observed that this map is null-homotopic if $\Lambda$ admits a semi-local generating function,
and proved the converse. They moreover showed that semi-local generating functions for $\Lambda$ can be classified up to stable semi-local equivalence
by homotopy classes of null-homotopies of $\gamma_\Lambda$. The latter set can be identified (non-canonically) with
$[L,\Z\times BO]$ in view of the Bott periodicity homotopy equivalence $\Omega (U/O) \simeq \Z\times BO$.

\begin{thm}[Giroux \cite{Giroux}, Latour \cite{Latour}] \label{thm:girouxlatour}
Let $M$ be a manifold and $\Lambda\subset J^1(M)$ a closed Legendrian submanifold.
\begin{enumerate}
\item
$\Lambda$ admits an excellent semi-local generating function if and only if the stable Gauss map
$\Lambda\to U/O$ is null-homotopic.
\item There is a free and transitive action of the group of (unbased) homotopy classes
of maps $[\Lambda,\Z\times BO]$ on the set of stable semi-local equivalence classes of excellent semi-local generating
functions for $\Lambda$.
\end{enumerate}
\end{thm}

A class in $[L,\Z\times BO]$ acts by stabilizing a generating function $f:U\to \R$
(where $U$ is a tubular neighborhood of $\Sigma_f\simeq L$) by a fiberwise non-degenerate quadratic form $q:U\times \R^n\to \R$
(the $\Z$-factor corresponding to the fiberwise signature of $q$ and the $BO$-factor corresponding to the fiberwise negative eigenspace of $q$).

\subsection{Spaces of generating functions}

Let $N$ be a manifold possibly with boundary and corners, such as $M\times\Delta^k$ where $M$ is a manifold without boundary
and $\Delta^k$ is the standard $k$-dimensional simplex. The space $J^1 N=T^* N \times \R$ is then also a manifold
with boundary and corners, with boundary stratification pulled back from the projection $\pi_N:J^1 N\to N$.

\begin{dfn}[Sets of Legendrian submanifolds]\label{dfn:setLeg}
We define $\LL(N)^*_0$ to be the set of all compact Legendrian submanifolds $\Lambda\subset J^1(N)$
such that $\del L=L\cap \pi_N^{-1}(\del N)$ and for each boundary stratum $N_i\subset \del N$,
$\Lambda$ is transverse to $\pi_N^{-1}(N_i)$ and the symplectic reduction of $\Lambda\cap \pi_N^{-1}(N_i)$
in $J^1 (N_i)$ is an embedded Legendrian submanifold.
\end{dfn}

For instance if $N=M\times [0,1]$ then $\LL(N)^*_0$ is the set of all compact Legendrian
cobordisms in $J^1(M\times[0,1])$ (in the sense of Arnol'd) between closed Legendrian submanifolds $\Lambda_0, \Lambda_1 \subset J^1 M$.

\begin{dfn}[Spaces of Legendrian submanifolds]\label{dfn:spaceLeg}
We define the simplicial sets $\LL(M)_\bullet$ and $\LL(M)_\bullet^*$ whose set of $k$-simplices
are respectively:
\begin{itemize}
\item $\LL(M)_k=\LL(M\times \Delta^k)_0^*$,
\item $\LL(M)_k^*$ is the subset of $\LL(M)_k$ consisting of all Legendrian submanifolds
which fiber over $\Delta^k$ and which reduce to a Legendrian submanifold $\Lambda_s$ in $J^1(M)$ for each $s\in \Delta^k$.
\end{itemize}
The structure maps are given by pullbacks, namely given $\Lambda \subset J^1(M\times \Delta^l)$
and a map $\alpha:\Delta^k\to \Delta^l$, we consider the induced map $\varphi_\alpha:J^1(M\times \Delta^l)\to J^1(M\times \Delta^k)$
and set $\alpha^* L=\varphi_\alpha(L)$.
The geometric realizations are denoted $\LL(M)$ and $\LL(M)^*$.
\end{dfn}
We can check that the simplicial sets $\LL(M)_\bullet$ and $\LL(M)^*_\bullet$ are well-behaved, namely they are Kan complexes.
The connected components of $\LL(M)^*$ are Legendrian isotopy classes of closed Legendrian submanifolds
in $J^1(M)$ while the connected components of $\LL(M)$ are Legendrian cobordism classes.

\begin{dfn}[Sets of generating functions]\label{dfn:setGF}
We denote $\GF(N)_0$ the set of all generating functions $f:N\times \R^n \to \R$ over $N$ with $n$ an even integer
where we identify $f$ with $\sigma^k f$ for all $k\geq 0$.
We denote $\GF(N)_0^*$ the subset of excellent ones and $\GF_\Lambda(N)_0^*$ the subset of those that generates a given
Legendrian submanifold $\Lambda\subset J^1 N$.
\end{dfn}

Let $M$ be a manifold without boundary and $\Lambda$ a Legendrian submanifold of $J^1 M$.

\begin{dfn}[Spaces of generating functions]\label{dfn:spaceGF}
We define the simplicial sets $\GF(M)_\bullet$, $\GF(M)^*_\bullet$ and $\GF_\Lambda(M)^*_\bullet$ whose sets of $k$-simplices are respectively:
\begin{itemize}
\item $\GF(M)_k=\GF(M\times \Delta^k)^*_0$,
\item $\GF(M)_k^*$ is the subset of $\GF(M\times \Delta^k)^*_0$ consisting of functions which are excellent
over each $M\times \{s\}$ for $s\in \Delta^k$ and induce Legendrian submanifold in $J^1(M\times \Delta^k)$ which fiber over $\Delta^k$,
\item $\GF_\Lambda(M)_k^*=\GF_{\Lambda\times \Delta^k}(M\times \Delta^k)^*_0$, where $\Lambda\times \Delta^k\subset J^1(M\times \Delta^k)=J^1(M) \times T^*\Delta^k$ is the product of $\Lambda$ with the zero-section of $T^*\Delta^k$.
\end{itemize}

The structure maps are given by pullbacks, namely for $\alpha : \Delta^k \to \Delta^l$ and $f\in \GF(M)_l$, we define $\alpha^* f \in \GF(M)_k$ by $(\alpha^* f)(x,t;v)=f(x,\alpha(t);v)$.

The geometric realization of these spaces are denoted simply $\GF(M)$, $\GF(M)^*$ and $\GF_\Lambda(M)^*$ respectively.

\end{dfn}

The simplicial sets $\GF(M)_\bullet$, $\GF(M)^*_\bullet$ and $\GF_\Lambda(M)^*_\bullet$ are Kan complexes.

\begin{rem}\label{rem:waldhausen}
The $\GF(\mathrm{point})$ can be thought of as the space of possible behaviour at infinity
for tame generating functions and it should be intimately related to Waldhausen's partition space $\P_\infty$ (see \cite{waldhausen_1982}).
In general $\GF(M)$ should be compared with the mapping space $\Map(M,\P_\infty)$.
\end{rem}

\begin{dfn}[Fibered diffeomorphism group]
We denote $\G(M)_0$ the group of diffeomorphisms of $M\times \R^n$ which lift the identity map of $M$
with even $n$ and where we identify $\varphi:M\times \R^n\to M\times \R^n$ with
$\varphi\times \id_{\R^2}:M\times \R^n\times \R^2 \to M\times \R^n\times \R^2$.
\end{dfn}

\begin{dfn}[Fibered diffeomorphism simplicial group]
We let $\G(M)_\bullet$ be the simplicial set whose set of $k$-simplices is $\G(M\times \Delta^k)_0$
and whose structure maps are given by pullbacks. We denote $\G(M)$ the geometric realization of $\G(M)_\bullet$.
\end{dfn}

\begin{prop}
The group $\G(M)$ is homotopy equivalent to $\Map(M,O)$ where $O=\colim_n O_n$ is the infinite orthogonal group.
\end{prop}
\begin{proof}
This follows from Alexander's trick: any fibered diffeomorphism $f(x,v)$ can be linearized by the formula $f(x,tv)/t$ for $t\in(0,1]$
and $\frac{\del f}{\del v}(x,0) v$ for $t=0$.
\end{proof}

We are interested in the homotopy type of $\GF_\Lambda(M)^*$. The connected components correspond
to the usual equivalence classes under stabilizations and fibered diffeomorphisms.

\begin{prop}\label{prop:pi_0=equiv}
Let $M$ be a manifold and $\Lambda$ a closed Legendrian submanifold of $J^1(M)$.
Two generating functions $f_0,f_1\in \GF_\Lambda(M)^*_0$ are stably equivalent if and only if they lie
in the same connected component of $\GF_\Lambda(M)^*$.
\end{prop}
\begin{proof}
Since $\GF_{\Lambda}(M)^*$ is a Kan complex, $f_0$ and $f_1$ are in the same connected component if and only if there
is an edge joining them, namely a generating function $f:M\times [0,1]\times \R^n\to \R$
such that $f(x,0,v)=f_0(x,v)$ and $f(x,1,v)=f_1(x,v)$.
If we have such a function $f=(f_t)_{t\in [0,1]}$, we apply Moser's trick and look for a fibered isotopy $(\varphi_t)_{t\in[0,1]}:M\times \R^n \to M\times \R^n$
generated by a vector field $X_t=a_t\frac{\del}{\del v}$ such that $f_t\circ \varphi_t=f_0$. This translates into the equation
\[\frac{d f_t}{dt} + a_t\frac{\del f_t}{\del v}=0.\]
Outside of a compact set of $\R^n$ this can be solved using the tameness property (see Proposition~\ref{prop:tamedeformation}).
Since each $f_t$ generates the same Legendrian submanifold $\Lambda$, the singular sets $\Sigma_{f_t}$
move by a fibered isotopy (explicitly $\Sigma_{f_t}=i_{f_t}^{-1}\circ i_{f_0}(\Sigma_{f_0}))$. By extending this fibered isotopy
we are reduced to the case where $\Sigma_{f_t}=\Sigma_{f_0}$ for all $t$. Now the function $\frac{d f_t}{d t}$ vanishes
along $\Sigma_{f_0}\times [0,1]$ because all $f_t$ generate the same Legendrian submanifold, and $\frac{\del f_t}{\del v}$
vanish transversely along $\Sigma_{f_0}$, so by Hadamard's division lemma we can find the required vector field $X_t$.

Conversely if $f_0$ and $f_1$ are equivalent, atfer stabilization there is a fibered diffeomorphism $\varphi_1$
such that $f_1\circ \varphi_1=f_0$. If $\varphi_1$ is isotopic to the identity map, namely if there is an edge
$\varphi_t$ in $\G(M)_1$ joining $\varphi_1$ and $\varphi_0=\id$, we can set $f_t=f_0\circ \varphi_t^{-1}$
to obtain an edge in $\GF_\Lambda(M)^*$ joining $f_0$ and $f_1$. We claim that $\varphi_1$ can be made isotopic to the identity
by stabilization as follows. First, we can assume that $\varphi_1(0)=0$ and then the Alexander trick:
\[\varphi_t(x,v)=\frac{\varphi_1(x,tv)}{t}\]
allows us to reduce $\varphi_1$ to its differential at the origin $\varphi_0(x,v)=d_{(x,0)}\varphi_1(0,v)$.
Now there is a smooth family of linear automorphism $(B_x)_{x\in M}:\R^{2k}\to \R^{2k}$ preserving the quadratic form
$h^k$ and such that the family $(\varphi_0)_x\oplus B_x$ is homotopic to the constant family equal to the identity map of $\R^n\times \R^{2k}$.
Hence $\sigma^k f_0$ and $\sigma^k f_1$ are equivalent by the fibered diffeomorphism $\varphi_1 \oplus B$ which is isotopic
to the identity as desired.
\end{proof}

There are natural (simplicial) maps
\[\GF(M)\to \LL(M) \quad \text{ and } \GF(M)^* \to \LL(M)^*\]
which map an excellent generating function to the Legendrian submanifold that it generates.
The persistence of generating functions under Legendrian isotopy translates into the following fibration statement.

\begin{thm}[Chekanov, Chaperon-Théret]\label{thm:spacelevelchekanov}
Let $M$ be a manifold without boundary. The map $\GF(M)^* \to \LL(M)^*$ is a Kan fibration onto its image.
In particular, if $\Lambda\subset J^1(M\times[0,1])$ is
a compact Legendrian cobordism which is the trace of a Legendrian isotopy $(\Lambda_t)_{t\in[0,1]}$,
then restriction at $t=0$ induces a homotopy equivalence
\[\GF_\Lambda(M\times [0,1])\to \GF_{\Lambda_0}(M).\]
\end{thm}

The map $\GF(M)^* \to \LL(M)^*$ is typically not surjective: its image consists of Legendrian submanifolds
that admit tame generating functions, and e.g. any stabilized Legendrian does not.
The fiber over a given $\Lambda\subset J^1 M$
is $\GF_\Lambda(M)^*$, its homotopy type is thus invariant under Legendrian isotopy
of $\Lambda$.

\begin{rem}
Generating functions in general do not extend over Legendrian cobordisms that are not induced by a Legendrian isotopy
and hence $\GF(M)\to \LL(M)$ is not a fibration.
\end{rem}

Let $M$ be a manifold and $\Lambda$ a Legendrian submanifold of $J^1(M)$ that properly maps to $M$.

\begin{dfn}
We define the set $\mu\GF(M)_0^*$ to be the quotient of the set of pairs $(f,\Sigma)$ where $f\in \GF(M)_0^*$,
$\Sigma \subset \Sigma_f$ is a proper codimension zero submanifold, where we identify two pairs $(f,\Sigma)$ and $(f',\Sigma')$
if $\Sigma=\Sigma'$, $i_f(\Sigma)=i_{f'}(\Sigma')$ and $f=f'$ near $\Sigma$.

We denote $\mu\GF_\Lambda(M)_0^*$ the subset of $\mu\GF(M)_0^*$ consisting of pairs $(f,\Sigma)$ such that $i_f(\Sigma)=\Lambda$.
\end{dfn}

\begin{dfn}\label{dfn:muGF}
We define the simplicial set $\mu \GF(M)^*_\bullet$ whose $k$-simplices
are elements $(f,\Sigma)$ of $\mu\GF(M\times \Delta^k)_0^*$ such that $\Sigma$ fibers over $\Delta^k$
and for each $s\in \Delta^k$, $i_{f_s} : \Sigma_s \to J^1(M)$ is a Legendrian embedding,
where we identify $(f,\Sigma)$ and $(f',\Sigma')$ if they coincide (stably) in a neighborhood of
$\Sigma=\Sigma'$.

We denote $\mu \GF_\Lambda(M)^*_\bullet$ the simplicial subset obtained by
restricting the set of $k$-simplices to $\mu\GF_\Lambda(M\times\Delta^k)_0^*$.
\end{dfn}

Again we can check that these simplicial sets are Kan complexes. The same proof as for Proposition~\ref{prop:pi_0=equiv}
gives the following.

\begin{prop}
Two elements $(f_0,\Sigma_0),(f_1,\Sigma_1)\in \mu\GF_\Lambda(M)^*_0$ are semi-locally stably equivalent if and only if they
lie in the same connected component of $\mu \GF_\Lambda(M)^*$.
\end{prop}

There is by construction a (simplicial) map
\[\mu\GF(M)^* \to \LL(M)^*\]
and the homotopy lifting property also holds in the semi-local case: this map is a Kan fibration.

The homotopy type of $\mu \GF_\Lambda(M)^*$ is invariant under Legendrian isotopy of $\Lambda$, and in fact
is completely determined by the following (enhanced version of) the theorem of Giroux and Latour.

\begin{thm}[Giroux-Latour]\label{thm:spacelevelgiroux-latour}
Let $\Lambda$ be a closed Legendrian submanifold of $J^1(M)$. The space $\mu GF_\Lambda(M)^*$
is non-empty if and only if the stable Gauss map $\Lambda \to U/O$ is null-homotopic. If non-empty,
$\mu GF_\Lambda(M)^*$ is a torsor over the space $\Map(L, \Z\times BO)$.
\end{thm}

There are natural maps from $\GF$ to $\mu\GF$. We can check that
the maps $\GF(M)^* \to \mu\GF(M)^*$ and $\GF_\Lambda(M)^*\to \mu \GF_\Lambda(M)^*$ are Kan fibrations.

We end this section with the following result which is a good sanity check for the definitions we use for spaces of generating functions.

\begin{prop}\label{prop:emptyleg}
Let $M$ be a manifold and $\Lambda$ the empty Legendrian submanifold of $J^1(M)$. Then $\GF_\Lambda(M)^*$ is contractible.
\end{prop}
\begin{proof}
The group $\G(M)$ acts transitively on $\GF_\Lambda(M)^*$. Indeed for any function $f:M\times \R^n\to \R$
generating the empty Legendrian, $f$ is fiberwise tame without any fiberwise critical points, so $(p,f):M\times \R^n \to M \times \R$ is a fibration. Its fibers are manifolds which become diffeomorphic to $\R^n$ after taking a product by $\R$, and hence after stabilization they are diffeomorphic to
$\R^{n-1}$ and we obtain that $f$ is stably equivalent to the projection $M\times \R^{n-1}\times \R \to \R$.

The stabilizer consists of fibered diffeomorphisms of $M\times \R^{n-1} \times \R$
which preserve the projection to $\R$. This subgroup is stably homotopy equivalent to the whole $\G(M)$ as the inclusion $O(n-1)\to O(n)$
becomes an equivalence in the colimit $n\to \infty$. Hence we obtain the contractibility of $\GF_\Lambda(M)^*$.
\end{proof}

\section{Sheaves}\label{sec:Sh}
\subsection{Recollections on sheaves}
For a topological space $X$ and stable presentable symmetric monoidal category $\mathcal{C}$, we write $\Sh(X, \mathcal{C})$ for sheaves in $X$ with coefficients in $\mathcal{C}$.
The basic operations of sheaf theory function well in this context \cite{Volpe}, as do, if $\mathcal{C}$ is compactly generated, the microlocal considerations \cite{Robalo-Schapira}.   When $\mathcal{C}$ is clear from context or fixed and irrelevant, we omit it from the notation.

For the present article, the main relevant case is when $\mathcal{C}$ is the $(\infty-)$category of spectra $Sp$.  There are various explicit descriptions of this category; modern foundations are given via an abstract characterization in \cite[Chap. 1]{Lurie-HA} (making precise the notion that a spectrum is a generalized homology theory).  Let us recall some main properties.  There is a map from the category of pointed spaces $\mathcal{S}$, sometimes denoted $\Sigma^\infty: \mathcal{S} \to Sp$.  We omit this notation: if we write a pointed space when a spectrum is expected, we mean to apply this map; for an unpointed space, we mean to first add a disjoint basepoint, and then apply this map.  The shift is given by $X[-1] := S^1 \wedge X$; if we write $\S$ for the zero sphere (with some chosen basepoint), then $\S[-n] = S^n$.
For $X \to Y$, the mapping cone provides an exact triangle $X \to Y \to \mathrm{Cone}(X \to Y) \xrightarrow{[-1]}$; in particular, when $X \subset Y$, then $X \to Y \to Y/X \xrightarrow{[-1]}$.  When $X \subset Y \subset Z$, one has (from the octahedron) $Y/X \to Z/X \to Z/Y \xrightarrow{[-1]}$.
There is a symmetric monoidal structure which descends from the smash product of pointed spaces, and has unit given by the zero-sphere $\S$; and there is an internal Hom.  Whenever we write Hom of spectra, we mean the internal Hom.

For the constant sheaf of sphere spectra $\S_X$ on a space $X$, one has, essentially by definition, $\Gamma(X, \S_X) = \Hom(X, \S)$.  More generally, for a map $f: X \to Y$ and open set $U \subset Y$,
one has
$(f_* \S_X)(U) = \Hom(f^{-1}(U), \S)$.
This dualization is the price of using sheaves rather than cosheaves, but loses no information if $X$ is a finite CW complex.  Indeed, these are well known to be dualizable,
$X \xrightarrow{\sim} \Hom(\Hom(X, \S), \S)$, as follows from the fact that, by definition, a finite CW complex admits a finite filtration whose associated graded pieces are pointed spheres.

We recall the notion of microsupport:  for a sheaf $F \in \Sh(M)$, there is a conical co-isotropic locus $ss(F) \subset T^*M$ measuring the failure of propagation of sections, as developped in \cite{Kashiwara-Schapira}.  For conic subsets $\Lambda \subset T^*M$, we write $\Sh_\Lambda(M)$ for the full subcategory of sheaves whose microsupport is contained in $\Lambda$.  Similarly, for $\Lambda \subset S^*M$, we write $\Sh_\Lambda(M)$ for the sheaves $F$ such that $ss(F) \cap S^*M \subset \Lambda$.

The following crucial result is the sheaf counterpart of Theorem~\ref{gen fn homotopy lifting 1}:

\begin{thm}[Guillermou-Kashiwara-Schapira \cite{GKS}] \label{sheaf homotopy lifting}
Let $\Lambda\subset S^*(M\times [0,1])$ be a compact Legendrian cobordism that is the trace of
a Legendrian isotopy $(\Lambda_t)_{t\in [0,1]}$.

Let $F_0\in \Sh_{\Lambda_0}(M;\k)$. Then there is $F\in \Sh_\Lambda(M\times [0,1])$ and an isomorphism $F|_{M\times \{0\}} \simeq F_0$.
Moreover, if we have $F,F'\in \Sh_\Lambda(M\times [0,1])$ and an isomorphism $u_0 : F|_{M\times \{0\}} \to F'|_{M\times \{0\}}$
then the isomorphism extends over $M\times [0,1]$.
\end{thm}

To deal with Legendrian submanifolds of the jet bundle $J^1 M$, we embed $J^1 M$ as the open subset of $S^*(M\times \R)$
consisting of covectors which are positive on the $\R$-factor.
In this context it is often convenient, following Tamarkin \cite{Tamarkin, Guillermou-Schapira-about-Tamarkin}, to consider the category
$$\Sh^+(M):=\Sh(M\times \R)/ \Sh_{T^*M \times T^-\R}(M \times \R),$$ where $T^- \R$ is the locus of nonpositive covectors.  This quotient admits an adjoint embedding it as the full subcategory  of $\Sh(M\times \R)$ on objects with non-negative microsupport and vanishing stalks at $M \times -\infty$.

\subsection{Microsheaves}

Let us recall the notion of microlocalization.  Following \cite[Chap. 6]{Kashiwara-Schapira}, consider the following
presheaf of ($\mathcal{C}$-linear) categories on $T^*M$:
$$\mu \Sh^{pre}(U) := \Sh(M) / \{F\, | \, ss(F) \subset T^*M \setminus U\}.$$
We write $\mu \Sh$ for its sheafification.  For $\Lambda \subset T^*M$ or $\Lambda \subset S^*M$, one has the corresponding  subsheaf of full subcategories
$\mu \Sh_\Lambda$ on objects (micro)supported in $\Lambda$; this sheaf of categories is (the pushforward of a sheaf) supported on $\Lambda$.

It is immediate from the definition that there is a natural map
$\Sh_\Lambda(M) \to \mu \Sh_\Lambda(\Lambda)$.  This map is not generally surjective.

It is shown already in \cite{Kashiwara-Schapira} that at a smooth Lagrangian/Legendrian point $p \in \Lambda$, there is a non-canonical isomorphism of $(\mu \Sh_\Lambda)_p$ with the coefficient category $\mathcal{C}$.  Fixing such an isomorphism, the composition
$$\Sh_\Lambda(M) \to \mu \Sh_\Lambda(\Lambda) \to (\mu \Sh_\Lambda)_p \cong \mathcal{C}$$
is termed a microstalk functor at $p$.

\begin{dfn}
    A sheaf $F \in \Sh_\Lambda(M)$ or microsheaf $F \in \mu \Sh_\Lambda(\Lambda)$ is said to be {\em simple} at a smooth Lagrangian/Legendrian point $p \in \Lambda$ if a microstalk functor at $p$ carries $F$ to an invertible object of $\mathcal{C}$.  If $\Lambda$ is a smooth Lagrangian or Legendrian, an element of $\Sh_\Lambda(M)$ or $\mu \Sh_\Lambda(\Lambda)$ is said to be {\em simple} if it is simple at all points of $\Lambda$.
\end{dfn}

It follows from the above discussion that if $\Lambda$ is Lagrangian/Legendrian, the sheaf of categories  $\mu \Sh_\Lambda$ is a twist of the sheaf of categories of local systems.  Abstractly, possible twistings are classified by maps $\Lambda \to B\Pic(\mathcal{C})$. Let us write $\mu_\Lambda: \Lambda \to B\Pic(\mathcal{C})$ for the specific twisting giving rise to $\mu \Sh_\Lambda$.  In \cite{Kashiwara-Schapira}, the degree $0$ part of $\mu_\Lambda$ was shown to be given by the Maslov index.  When $\mathcal{C}$ is the category of modules over a (discrete) ring $R$, $\Pic(\mathcal{C})$ is truncated in degree 1, and Guillermou described $\mu_\Lambda$ in its entirety \cite{Guillermou}. In particular:

\begin{thm}[Guillermou \cite{Guillermou}]
Take $\mathcal{C} = \Z-mod$, and fix a closed Legendrian submanifold $\Lambda \subset S^*(M)$.
Then $\mu \Sh_\Lambda(\Lambda)$ contains a simple object iff the first Maslov class $\mu_1(\Lambda) \in H^1(\Lambda;\Z)$ and a relative Stiefel Whitney class
$rw_2(\Lambda) \in H^2(\Lambda;\Z/2)$ both vanish. These classes correspond to the composition
of the Gauss map $\Lambda\to U/O$ with the determinant $B(\Z\times BO) \to B(\Z\times B(\Z/2))$ under
the Bott periodicity isomorphism $U/O \simeq B(\Z\times BO)$.

Isomorphism classes of simple microlocal sheaves for $\Lambda$ form a torsor over
\[[\Lambda, \Z\times B(\Z/2)]\simeq H^0(\Lambda;\Z)\oplus H^1(\Lambda;\Z/2).\]
\end{thm}

More generally, $\mu_\Lambda$ is determined by the universal case when $\mathcal{C}$ is the category of spectra, where it was calculated by Jin:

\begin{thm}[Jin \cite{Jin}]
A closed Legendrian submanifold $\Lambda \subset S^*(M)$ admits a microlocal sheaf over $\S$
if and only if the composition of the Gauss map $\Lambda\to U/O$ with the $J$-homomorphism
\[U/O\simeq B(\Z\times BO) \to B(\Z\times BG)\]
is null-homotopic.

Microlocal equivalence classes of microlocal sheaves for $\Lambda$ form a torsor over
\[[\Lambda, \Z\times BG].\]
\end{thm}

\subsection{Spaces of sheaves}

Let us recall that we may form a topological space from an ordinary (not $\infty$) category, called its nerve, as follows: the objects become 0-simplices, the 1-morphisms become 1-simplices, compositions become 2-simplices, etc.  For $(\infty,1)$-categories (i.e. all categories we consider here), the description is simpler: such a category is by definition a simplicial set with certain properties; the nerve is just the underlying simplicial set.  Given an $(\infty, 1)$-category $\mathcal{C}$, we write $\mathcal{C}^*$ for the $(\infty,0)$-category obtained by forgetting all non-invertible morphisms, and understand it as the homotopy type of a space.

When working with categories of sheaves, microsheaves, etc., it is often convenient to use a more internal description.  For specificity let us give it for $\Sh(M)^*$: the  $n$-simplices are sheaves on $\Delta^n \times M$ whose microsupport is contained in $T^*_0(\Delta^n) \times T^*M$, which hence are locally constant in the $\Delta$ direction.  Thus 0-simplices are just sheaves on $M$; and from a 1-simplex, you can reconstruct a morphism by parallel transport in the $\Delta^1 = [0,1]$ direction.

This identification allows many theorems which are originally stated for e.g. isomorphism classes of sheaves to be promoted to assertions about the homotopy type $\Sh(M)^*$ by simply applying the original theorem as stated to all $M \times \Delta^n$.  (See e.g. \cite{li-nadler-shende} for  discussions of a similar nature.)
In particular, we have the following:

\begin{thm}\label{thm:spaceleveljin} (Space level version of \cite{Jin}.)
Let $\Lambda$ be a closed Legendrian submanifold of $J^1(M)$ and $\gamma_\Lambda:\Lambda\to U/O$ its stable Lagrangian Gauss map.
The space $\mu\Sh^+_\Lambda(M)^*$ is non-empty if and only if the composition $J\circ \gamma_\Lambda : L\to B(\Z\times BG)$ is null-homotopic.
If non-empty the space $\mu \Sh^+_\Lambda(M)^*$ is a torsor over the space $\Map(L,\Z\times BG)$.
\end{thm}

\begin{thm}\label{thm:spacelevelGKS} (Space level version of \cite{GKS}.)
The map $\Sh(M)^*\to \LL(M)^*$ is a Kan fibration. In particular, if $\Lambda\subset J^1(M\times[0,1])$ is
a compact Legendrian cobordism which is the trace of a Legendrian isotopy $(\Lambda_t)_{t\in[0,1]}$,
then restriction at $t=0$ induces a homotopy equivalence
\[\Sh_\Lambda(M\times [0,1])^*\to \Sh_{\Lambda_0}(M)^*.\]
\end{thm}

\subsection{The sheaf of sublevel set cohomotopies}\label{sec:GFtoSh}

Let $f:M\times \R^n\to \R$ be a function. Consider the epigraph
\[\mathrm{Above}(f)=\{(x,y,z)\in M\times \R^n\times \R, f(x,y)\leq z\}.\]
For $U \subset M \times \R$ we write
$$\mathrm{Above}(f)(U) := \{(x,z) \in U,\, y \in \R^n, f(x,y)\leq z\}.$$
It is obvious that if $\phi$ is a fiberwise diffeomorphism of $M \times \R^n \to \R^n$, then
$\phi$ induces diffeomorphisms
$\mathrm{Above}(f)(U) \cong \mathrm{Above}(f \circ \phi)(U)$.

Denoting the projection $q:M\times \R^n\times \R\to M\times \R$, we introduce the sheaf
$q_* \S_{\mathrm{Above}(f)}$ on $M \times \R$.  Essentially by definition,
$$\Gamma(U, q_* \S_{\mathrm{Above}(f)}) = \Hom(\mathrm{Above}(f)(U) ,\S)$$

For $V \subset M$ and $a: V \to \R$, we write
$$V \rtimes (-\infty, a] := \{(x, z) \in V \times \R\,|\, z \le a(x)\}.$$
Similarly for functions $a \le b$, we write
$$V \rtimes (a, b] := \{(x, z) \in V \times \R\,|\, a(x) < z \le b(x)\}.$$
We will be
interested in the
pointed space
$$\mathrm{Above}(f)(V; (a, b] ) :=
\mathrm{Above}(f)(V \rtimes (-\infty, b]) / \mathrm{Above}(f)(V \rtimes (-\infty, a])$$
Note the corresponding spectrum is encoded by
$q_* \S_{\mathrm{Above}(f)}$:
\begin{align} \label{Above f sheaf versus space}
\nonumber & \Gamma(V \ltimes (a, b]; q_* \S_{\mathrm{Above}(f)}) \\
\nonumber & =
\mathrm{Cone}(\Gamma(V \rtimes (-\infty, b]; q_* \S_{\mathrm{Above}(f)}) \to \Gamma(V \rtimes (-\infty, a]; q_* \S_{\mathrm{Above}(f)}))  \\
\nonumber & = \mathrm{Cone}( \Hom(\mathrm{Above}(f)(V \rtimes (-\infty, b]), \S )  \to \Hom ( \mathrm{Above}(f)(V \rtimes (-\infty, a]), \S))\\
\nonumber
& =  \Hom(\mathrm{Cone}(\mathrm{Above}(f)(V \rtimes (-\infty, a]) \to \mathrm{Above}(f)(V \rtimes (-\infty, b]) ) , \S)
\\
& = \Hom(\mathrm{Above}(f)(V; (a, b]), \S)
\end{align}

The stalks near $M\times \{-c\}$ for large $c\gg 0$ may be nontrivial, but assuming that $f$ is fiberwise tame, the restriction to $M\times\{-c\}$ is a locally constant
sheaf $F_-$, and there is a natural map  $u_-: q_* \S_{\mathrm{Above}(f)} \to F_-\boxtimes \S_\R$.

\begin{dfn} \label{associated sheaf}
We write $\s(f) := Cone(u_-)[n/2]$.
\end{dfn}

The use of $Cone(u_-)$ is more invariantly expressed in terms of Tamarkin's category $\Sh^+(M)$, which is the quotient of sheaves on $M \times \R$ by those sheaves with negative microsupport \cite{Tamarkin}.  This quotient carries a section $\Sh^+(M) \to \Sh(M \times \R)$ with image in the sheaves with nonnegative microsupport.  For sheaves which already have non-negative microsupport and are constant in a neighborhood of $M \times - \infty$, the composition of the quotient and the section is well known to be given by the cone on this $u_-$.  Thus $\s(f)$ is obtained by taking the image of $q_* \S_{\mathrm{Above}(f)}$ in $\Sh^+(M)$, after re-embedding
$\Sh^+(M) \hookrightarrow \Sh(M \times \R)$.

As we have already noted in the introduction,  if we fix the embedding $J^1(M) \subset T^*(M \times \R)$ as the locus where the $\R$ codirection has unit length, then
$ss(\s(f)) \cap J^1(M)$ is the Legendrian generated by $f$.
\begin{lem} \cite{Viterbo-introduction} \label{gen fn sheaf support}
    Let $E \to M$ be a bundle and let $f : E \to \R$ be a fiberwise tame function.  Then $ss(\mathbf{s}(f))|_{J^1 M} \subset \Lambda_f$, with equality when the fiberwise critical locus is transversely cut out (namely $f$ is a generating function).
\end{lem}
\begin{proof}
    The microsupport of $\S_{f \le z}$ is the upward conormal to the graph of $f$.
    The standard estimate of the microsupport of a pushforward along $\pi$ require properness of $\pi$, but tameness of $f$ is a good enough substitute, since then we may extend $f$ to a disk bundle while remaining a fibration near the new boundary.

    The pushforward estimate \cite[Prop. 5.4.4]{Kashiwara-Schapira} asserts that the microsupport of the pushforward along $\pi$ is contained in the locus obtained by applying the Lagrangian correspondence given by the conormal of $\pi$.  This is exactly how we described the Legendrian determined by a generating function above.  The fact that this containment is an equality is readily checked by hand.
\end{proof}

We have
\begin{equation}
\label{relative sections of Sf}
\Gamma(V \rtimes (a, b]; \s(f)) =
\Gamma(V \rtimes (a, b]; q_* \S_{\mathrm{Above}(f)}[n/2])
= \Hom(\mathrm{Above}(f)(V; (a, b]), \S)[n/2].
\end{equation}

\begin{lemma}
For any excellent generating function $f$, the sheaf $\s(f)$ is simple.
\end{lemma}
\begin{proof}
    This follows readily from \eqref{relative sections of Sf}.
\end{proof}

\begin{example}
    Consider  $M = \mathrm{point}$, $n=1$, and the functions $f= \pm y^2: \R_y \to \R_z$.  Because $q_* \S_{\mathrm{Above}(y^2)}$ is $\S_{[0, \infty)}$, and so $\s(y^2) = \mathrm{Cone} (\S_{[0, \infty)} \to 0 )[1/2] = \S_{[0, \infty)}[-1/2]$.  Meanwhile, $q_* \S_{\mathrm{Above}(-y^2)}$ is a sheaf whose stalk is $\S \oplus \S$ for $z < 0$ and $\S$ for $z \ge 0$; one checks $\s(-y^2) = \S_{[0, \infty)}[1/2]$.

    Similarly, for general $M$ and $f= \pm y^2: M \times \R_y \to \R_z$, we have
    $$\s(-y^2) = \S_{M \times [0, \infty)}[1/2] \qquad \qquad \s(y^2) = \S_{M \times [0, \infty)}[-1/2]$$
\end{example}

For sheaves on $\R \times \R$, writing $a: \R \times \R \to \R$ for addition,
there is the convolution product given by
$F \star G :=  a_* (F \otimes G)$.  Similarly, for $F, G$ sheaves on $M \times \R$, there is a convolution product given by tensoring in the $M$ factor, and convolving in the $\R$ factor.  (These products are studied in e.g. \cite{Tamarkin, Guillermou-Schapira-about-Tamarkin}.)
Convolution restricts to the category of sheaves with nonnegative microsupport; on this category, the monoidal unit is $\S_{M \times [0, \infty)}$.

For generating functions
$f: M \times \R^n \to \R$ and $g: M \times \R^m \to \R$, we write $f \oplus g: M \times \R^n \times \R^m \to \R$ for the sum of the compositions of $f$ and $g$ with the corresponding projections.

\begin{lemma} \label{s and sum}
    $\s(f\oplus g)\simeq \s(f)\ast\s(g).$
\end{lemma}
\begin{proof}
    The corresponding assertion
    $q_* \S_{Above(f)} \star q_* \S_{Above(g)} = q_* \S_{Above(f\oplus g)}$
    is immediate from the definitions.
    One can deal with the cone by $u_-$ either by hand (using the tameness) or instead by recalling that the convolution passes to Tamarkin's quotient $\Sh^+$.
\end{proof}

\begin{corollary} \label{stabilization effect on sheaf}
    $\s(f \oplus -y^2) = \s(f)[1/2]$, and $\s(f \oplus y^2) = \s(f)[-1/2]$.
\end{corollary}

In particular $\s(f)$ is invariant by stabilization : $\s(f\oplus xy)=\s(f)$.

\begin{remark}
The monoidal structure $\star$ admits an internal Hom, for which one has
$\mathcal{H}om^\star(\s(f),\s(g)) = \s(-f \oplus g)$.
\end{remark}

Finally, as we have defined a semilocal generating function for $\Lambda$ as a function whose microsupport contains $\Lambda$ as a closed and open subset (Def. \ref{semilocal gf def}), we may restrict the microlocalization of such a function to $\Lambda$ and obtain a microsheaf on $\Lambda$. Since we have defined equivalence of semilocal generating functions by asking for fibered diffeomorphism near the critical locus after stabilization, the resulting microsheaf depends only on the equivalence class of the semilocal generating function.

The map $\mathbf{s}$ also determines a map between spaces of generating functions and sheaves by applying it to $M \times \Delta^n$ for all simplices $\Delta^n$.   We have now explained \eqref{eq:mainsquare}.

\section{The doubling trick}\label{sec:doubling}

\subsection{Doubling and reduction}\label{gen fn doubling}

Consider the function $D:\R^2\to \R$ defined by
\[D(t,w)=w^3-3tw=w(w^2-3t)\]
It has a single critical point at $t=0, w=0$ of index $1$ and with critical value $0$.
In fact it is globaly conjugate to the quadratic form $h(u,v)=uv$ via the diffeomorphism of $\R^2$
given by
\[u=w^2-3t, v=w.\]
In particular, $D$ is a tame function.

Now let us think of $D$ as a family of functions $(D_t)_{t\in\R}:\R\to \R$. We observe that $\frac{\del D}{\del w}(t,w)=3(w^2-t)$ vanishes transversely
along the parabola $\Sigma_D=\{w^2=t\}$. So $D_t$ has no critical points for $t<0$
and two critical points for $t>0$ with critical values $\pm 2 t^{3/2}$. Denote $\Lambda_D \subset J^1(\R)$
the Legendrian submanifold generated by $D$. In order for $D$ to be a genuine generating function for $\Lambda_D$
we need to check the tameness condition.
\begin{lem}
$D$ is a fiberwise tame function (with respect to the projection $p:(t,w)\mapsto t$).
\end{lem}
\begin{proof}
The definition of fiberwise tameness is readily checked with any function $a:\R\to (0,+\infty)$ such that $a(t)>2t^{3/2}$ for $t>0$ and $A=\{(t,w), |D(t,w)|\leq a(t)\}$.
\end{proof}

Given a Legendrian submanifold $\Lambda\subset J^1(M)$, we can form a product Legendrian
$\Lambda\times \Lambda_D\subset J^1(M\times \R)$ whose reduction under the projection $M\times \R\to M$
is $\Lambda$ and whose reduction at slices $\{t=t_0>0\}$ are $2$-copies of $\Lambda$ (namely $\Lambda$ shifted by $\pm 2t_0^{3/2}$).

Given a generating function $f:M\times \R^n\to \R$ for $\Lambda\subset J^1(M)$, the function $g:M\times \R\times \R^n \times \R \to \R$ given by
\[g(x,t;u,w)=f(x;u)+D(t,w)\]
is fiberwise tame (with respect to the projection to $M\times \R$) per Lemma~\ref{lem:tamesum}, and generates $\Lambda\times \Lambda_D\subset J^1(M\times \R)$.
This defines a (simplicial) map
\[\DD:\GF_\Lambda(M)^*\to \GF_{\Lambda\times \Lambda_D}(M\times \R)^*.\]
Conversely if we have a generating function $g:M\times \R\times \R^n\to \R$ for $\Lambda\times \Lambda_D \subset J^1(M\times \R)$, we get a generating function $f$
for $\Lambda$ by reduction (i.e. viewing the $t$-coordinate as an auxiliary variable):
\[f(x;t,u)=g(x,t;u).\]
The tameness condition for $f$ follows from Lemma~\ref{lem:tamepushforward} applied to the fibration $q:M\times \R\to M$ given by $q(x,t)=x$.
Indeed the properness assumptions are met: the singular set of $g$ properly maps to $M\times \R$ (under the map $(x,f)$) and the singular set of $f$ is the intersection of the singular set of $g$ with the slice $\{t=0\}$ and hence properly maps to $M$.

So we have a (simplicial) map
\[\RR:\GF_{\Lambda\times \Lambda_D}(M\times \R)^*\to \GF_\Lambda(M)^*.\]

The composition $\RR\circ \DD$ transforms a generating function $f:M\times \R^n\to \R$ for $\Lambda$ into another one $F:M\times \R\times \R^n\times \R$ given by
\[F(x;t,u,w)=f(x;u)+D(t,w)\]
which is equivalent to the stabilization $\sigma f$ since $D$ is conjugate to the quadratic form $h=uv$. Hence we have:

\begin{proposition} \label{prop:RD}
The composition $\RR\circ \DD : \GF_\Lambda(M)^*\to \GF_\Lambda(M)^*$ is homotopic to the identity map.
\end{proposition}

The above discussion is equally valid for semi-local generating functions, so we also obtain:
\begin{proposition}
The composition $\RR\circ \DD : \mu\GF_\Lambda(M)^*\to \mu\GF_\Lambda(M)^*$ is homotopic to the identity map.
\end{proposition}

\subsection{Rotation}

We want to show that the other composition $\DD\circ \RR$ is also homotopic to the identity. In view of Proposition~\ref{prop:RD}
it is equivalent to show $\RR$ is a homotopy equivalence.
For this purpose we use a rotation of $T^* \R$ of angle $\pi/2$ in order to map the cusped Legendrian $\{z=\pm 2t^3/2\}$
to a smooth graph $\{z=3t^3\}$. This graph has a generating function without auxiliary variable $C(t)=3t^3$, we denote $\Lambda_C$ the corresponding
Legendrian submanifold of $J^1(\R)$.
For $f:M\times \R \times \R^n \to \R$, we define
$h:M\times \R \times \R^n \times \R \to \R$ by the formula
\[h(x,t,v,s)=f(x,s,v)+st.\]
If $f$ is fiberwise tame with respect to the projection $(x,s,v)\mapsto (x,s)$, then $f(x,s,v)+ut$ is tame with respect to the projection $(x,s,v,u,t)\mapsto (x,s,u,t)$
by Lemma~\ref{lem:tamesum} and then $h$ is fiberwise tame with respect to $(x,t,v,s)\mapsto (x,t,s)$ by Lemma~\ref{lem:tamepullback} using pullback along the diagonal $u=s$.
Finally if $f$ generates $\Lambda\times \Lambda_D$, then $h$ is fiberwise tame with respect to $(x,t,v,s)\mapsto (x,t)$ by Lemma~\ref{lem:tamepushforward}
(the properness conditions are met) and generates $\Lambda\times \Lambda_C$.
Hence we obtain a map
\[\A:\GF_{\Lambda\times \Lambda_D}(M\times \R)^*\to \GF_{\Lambda\times \Lambda_C}(M\times \R)^*.\]
Similarly we define
\[\A':\GF_{\Lambda\times \Lambda_C}(M\times \R)^*\to \GF_{\Lambda\times \Lambda_D}(M\times \R)^*\]
by the formula
\[\A'(h)(x,t,v,s)=f(x,s,v)-st.\]

\begin{proposition}\label{prop:Aequiv}
The compositions $\A\circ \A'$ and $\A'\circ \A$ are homotopic to the identity. In particular, $\A$ is a homotopy equivalence.
\end{proposition}
\begin{proof}
Given a function $f(x,t;v)\in \GF_{\Lambda\times \Lambda_D}(M\times \R)^*$, the function $f'=\A'\circ \A(f)$ writes:
\[f'(x,t;v,s,u)=f(x,s,v)+su-ut.\]
We make a first change of variables $s'=s-t$ and use Hadamard's division Lemma to write
$f(x,s'+t,v)=f(x,t,v)+s'g(x,t,v,s')$ for some smooth function $g$, so that
\[f''(x,t;v,s',u)=f(x,t,v)+s'g(x,t,v,s')+s'u.\]
Then we set $u'=u+g(x,t,v,s')$ and obtain
\[f'''(x,t;v,s',u')=f(x,t,v)+s'u'\]
which is a stabilization of $f$. We can carry out all steps by a continuous
deformation from $f'$ to $f'''$, and moreover this works parametrically, so $\A'\circ\A$
is homotopic to the identity. The proof works verbatim for $\A\circ \A'$.
\end{proof}

Next we define
\[\II_0:\GF_{\Lambda\times \Lambda_C}(M\times \R)^*\to \GF_\Lambda(M)^*\]
the map induced by restriction to the slice $\{t=0\}$.
The key property is the following.

\begin{proposition}\label{prop:R=I0A}
We have $\RR=\II_0\circ \A$.
\end{proposition}
\begin{proof}
It is a direct computation:
\[\II_0\circ \A(f)(x,s,v)=\A(f)(x,0,s,v)=f(x,s,v)+0\times s= f(x,s,v).\]
\end{proof}

\begin{proposition}\label{prop:I0}
The map $\II_0$ is a homotopy equivalence.
\end{proposition}
\begin{proof}
Since $\Lambda_C$ is graphical over $\R$, all functions $(x,v)\mapsto f(x,t,v)-C(t)$ generate the same Legendrian $\Lambda$
and thus $\GF_{\Lambda\times \Lambda_C}(M\times \R)^*$ can be identified with the space of maps from $\R$ to $\GF_\Lambda(M)^*$, and hence
 restriction at $t=0$ is a homotopy equivalence.
\end{proof}
We conclude:

\begin{thm}\label{thm:GF_RD}
The maps $\RR$ and $\DD$ are inverse homotopy equivalences between the spaces $\GF_\Lambda(M)^*$ and $\GF_{\Lambda\times \Lambda_D}(M\times \R)^*$.
\end{thm}
\begin{proof}
By Proposition~\ref{prop:Aequiv}, Proposition~\ref{prop:I0} and Proposition~\ref{prop:R=I0A}, we learn that the map $\RR$ is
a homotopy equivalence. Now from Proposition~\ref{prop:RD}, we have $\RR\circ \DD \sim \id$, so composing with a homotopy inverse of $\RR$
we get $\DD\circ \RR \sim \id$.
\end{proof}

This also holds at the level of semi-local generating functions.
\begin{thm}\label{thm:muGF_RD}
The maps $\RR$ and $\DD$ induce inverse homotopy equivalences between the spaces $\mu\GF_\Lambda(M)^*$ and $\mu\GF_{\Lambda\times \Lambda_D}(M\times \R)^*$.
\end{thm}

Finally we observe that for $t>0$, $\Lambda_D$ is (multi-)graphical over $\R$, and hence a semi-local generating function defined
for $t<\epsilon$ with $\epsilon >0$ can be extended to the whole $\R$ in a homotopically unique way. So setting
\[\GF^{\leq 0}_{\Lambda\times \Lambda_D}(M\times \R)^*:=\colim_{\epsilon >0} \GF_{\Lambda\times \Lambda_D}(M\times (-\infty,\epsilon))^*,\]
and similarly for the semi-local version with $\GF$ replaced by $\mu \GF$,
we have:
\begin{prop}\label{prop:muleq0}
The restriction map $\mu\GF_{\Lambda\times \Lambda_D}(M\times \R)^* \to \mu\GF^{\leq 0}_{\Lambda\times \Lambda_D}(M\times \R)^*$ is a homotopy equivalence.
\end{prop}
Note that for a given closed Legendrian $\Lambda$, the above colimit is actually achieved for some small enough $\epsilon>0$.

\subsection{From semi-local to global via doubling}

We now turn to the main result of this section which allows to pass from semi-local generating functions
for a Legendrian $\Lambda$ to genuine generating functions for $\Lambda\times \Lambda_D$ over $M\times(-\infty,\epsilon)$ for small $\epsilon >0$.

\begin{thm}\label{thm:GFdoubling}
The map $\GF^{\leq 0}_{\Lambda\times \Lambda_D}(M\times \R)^* \to \mu\GF^{\leq 0}_{\Lambda\times \Lambda_D}(M\times \R)^*$ is a homotopy equivalence.
\end{thm}
\begin{proof}
We first consider the case where $\Lambda$ is empty. In this case, the right hand side is reduced to a single point so we need to prove
that the left hand side is contractible.
Since $\Lambda$ is empty, the restriction map at $t=0$ (or at any $t=t_0$ for $t_0\leq 0$) $\GF^{\leq 0}_\emptyset(M\times \R) \to \GF_\emptyset(M)$
is a homotopy equivalence. The result now follows from Proposition~\ref{prop:emptyleg}.

Another extreme case is when $M$ is a single point and $\Lambda$ consists of finitely many points $z_1,\dots,z_k$ in $\R=J^1(*)$.
Let $\phi\in \mu\GF^{\leq 0}_{\Lambda\times \Lambda_D}(\R)^*$, this is a germ of function $\phi:\R^n\times \R \to \R$
near finitely many points $(v_1,0),\dots,(v_k,0)$. For each of these points, we pick an embedding $\psi_i:D^{n-1} \times [z_i-a,z_i+a]\times[-\epsilon,\epsilon]\to \R^n\times \R$ whose image is contained in the domain of definition of $\phi$ ($a>0$ and $\epsilon>0$ can be taken arbitrarily small) and satisfies:
$\psi_i(0,z_i,0)=(v_i,0)$, $\psi_i(x,z,t)=(j(x,z,t),t)$, $\phi(\psi_i(x,z,-\epsilon))=z$, $\phi(\psi_i(x,z_i\pm a,t))=z_i\pm a$
and $\phi(\psi_i(x,z,t))=z$ for $|x|=1$. Denote $K_i$ the image of $\psi_i$.
We need to prove that the space of functions $f:\R^n\to \R$ which coincide with $\phi$ on each $K_i$ is highly connected (thus contractible in the limit $n\to \infty$).
We claim that restriction at $t=-\epsilon$ yields a homotopy equivalence between this space and the space of functions
$f:\R^n\to \R$ which have no critical points and agree with $\phi$ on $K_i\cap\{t=-\epsilon\}$, namely $f(\psi(x,z,-\epsilon))=z$.
Indeed by a suitable retraction of $\{t\leq \epsilon\}\subset \R^n\times \R$ onto $\{t\leq -\epsilon\}\cup \cup_i K_i$, we can deform
such functions until they coincide on $\{t\leq \epsilon\}$ provided they already coincide in $\{t\leq-\epsilon\}\cup \cup_i K_i$.
Finally we need to argue that the space $F$ of functions $\R^n\to \R$ without critical points which coincide with the projection $D^{n-1}\times [z_i-a,z_i+a]\to [z_i-a,z_i+a]$ in finitely many such disjoint cylinders embedded in $\R^n$, is highly connected when $n$ goes to infinity.
We claim first that this space is equivalent to the space of functions without critical points $f$ such that $\nabla f(v_i)=dz$.
Indeed, if $\nabla f (v_i)=dz$, we can deform slightly $f$ so that $f=z$ near $v_i$ (the center of $K_i$), then by dilating $K_i$ we can arrange $f=z$
on the whole $K_i$. From the space $E$ of all functions $\R^n\to \R$ without critical points, the collection of $\nabla f(v_i)$ thus defines a map to $(\R^n\setminus\{0\})^k\simeq (S^{n-1})^k$ whose fiber is homotopy equivalent to the space $F$.
The space $(S^{n-1})^k$ is $(n-2)$-connected and $E$ is also highly connected (see Proposition~\ref{prop:emptyleg}), therefore $F$ is highly connected.

We now turn to the general case. We consider the fiber over some semi-local generating function $\phi$ for $\Lambda\times \Lambda_D$,
namely the space of generating functions $f:M\times \R\times \R^n\to \R$ for $\Lambda\times \Lambda_D$ over $M\times (-\infty, \epsilon)$
which agree with $\phi$ near $\Sigma_\phi=\Sigma_f$. Denote $\Sigma_0=\Sigma_\phi\cap\{t=0\}$ and $\phi_0$ the restriction of $\phi$ to $\Sigma_0$.
We pick a tubular neighborhood $U\times [-a,a]$ of $\Sigma_0$ (denote $r:U\to \Sigma_0$ the corresponding projection) and
an embedding $\psi:U\times[-a,a]\times[-\epsilon,\epsilon]\to M\times \R^n\times \R$ such that
$\psi$ is the identity on $\Sigma_0\times\{0\}$, $\psi(u,z,t)=(*,*,t)$, $f(\psi(u,z,-\epsilon))=\phi_0(r(u))+z$ and $f(\psi(u,z,t))=\phi_0(r(u))+z$ for $u\in \del U$.
As in the previous case, the function $f$ is then determined up to contractible choice by its restriction to $t=-\epsilon$. The gradient $\nabla f$
along $\Sigma_0$ defines a map from the space of functions $f:M\times \R^n\to \R$ fiberwise without critical points to
the space $\Map(\Sigma_0, S^{n-1})$, both of which are highly connected, with fiber homotopy equivalent to the space we are considering, so the claim follows.
\end{proof}

\subsection{Doubling for sheaves} \label{sec: sheaf doubling}
We turn to the sheaf setting, where the doubling trick was first employed by Guillermou in \cite{Guillermou} (see also the exposition in \cite[Sec. 7]{Nadler-Shende}).

The construction is guided by the fact that it should be compatible with the map $\s$ from generating function to sheaves.
Let $D: \R^2 \to \R$, $D(t,w) = w^3-3tw$ be the function from Subsection~\ref{gen fn doubling}. As there, we regard it as a generating function over the $t$ line, so we have some $\s(D) \in \Sh^+(\R)$. After embedding $\Sh^+(\R_t) \hookrightarrow \Sh(\R_t \times \R_z)$, $\s(D)$ is a sheaf on $\R^2$ with positive microsupport in the $z$ direction.

Recall that, in general, there is a product $\star \, : \Sh^+(M) \times \Sh^+(N) \to \Sh^+(M \times N)$ given, after the embeddings $\Sh^+(M) \to \Sh(M \times \R)$ etc, by external tensor product in the $M, N$ factors and convolution in the $\R$ factors.

We define
\begin{align*}
    \DD: \Sh^+(M) & \to \Sh^+(M \times \R) \\
    F & \mapsto F \star \s(D)
\end{align*}
The standard estimate on pushforwards of microsupport shows that, as for generating functions, $ss(F \star \s(D)) \subset ss(F) \times \Lambda_D$.
So $\DD$ restricts to a map
\[\DD:\Sh^+_\Lambda(M)^*\to \Sh^+_{\Lambda\times \Lambda_D}(M\times \R)^*.\]

We define
$\mathcal{R}: \Sh^+(M \times \R_t) \to \Sh^+(M)$ as pushforward along the $\R_t$ factor.  The following are straightforward 6-functor manipulations:

\begin{lemma} \label{sheaf r after d}
    $\mathcal{R}  \circ \mathcal{D} = 1$
\end{lemma}

To show that $\DD\circ \RR$ is also the identity, we can use the rotation operator $\A$
as we did for generating functions. Consider the function $A:\R^2\to \R$ defined by $A(s,t)=st$
and $\s(a)\in\Sh^+(\R\times \R)$ (it is a quantization of the rotation of angle $\pi/2$ of $\R^2$).
This kernel acts on $\Sh^+(M\times \R)$ and defines a map
\[\A:\Sh^+_{\Lambda\times \Lambda_D}(M\times \R)^*\to \Sh^+_{\Lambda\times \Lambda_C}(M\times \R)^*\]
Using the inverse rotation, we see that $\A$ is an equivalence.

Finally we have the restriction at $t=0$ operator $\II_0:\Sh^+(M\times \R)\to \Sh^+(M)$.

\begin{lemma}
We have
$\RR\circ \A=\II_0.$
\end{lemma}
Moreover $\II_0$ induces an equivalence $\Sh^+_{\Lambda\times \Lambda_C}(M\times \R)^* \to \Sh^+_\Lambda(M)^*$
since $\Lambda_C$ is graphical over $\R$.
So we obtain

\begin{thm}\label{thm:Sh_RD}
The maps $\RR$ and $\DD$ are inverse homotopy equivalences between $\Sh_{\Lambda\times \Lambda_D}(M\times \R)^*$ and $\Sh_\Lambda(M)^*$.
\end{thm}

\begin{lemma} \label{s intertwines}
    The map $\s$ intertwines the $\mathcal{R}, \mathcal{D}$ for generating functions with those for sheaves.
\end{lemma}
\begin{proof}
    For $\mathcal{D}$, this is a special case of Lemma \ref{s and sum}.  For $\mathcal{R}$, this follows from functoriality of pushforward.
\end{proof}

Finally, the sheaf counterpart of Theorem~\ref{thm:GFdoubling} is the following result of Guillermou.

\begin{theorem}  (\cite{Guillermou}, see also \cite{Nadler-Shende} for this  formulation) \label{thm:Shdoubling}
    Let $\Lambda \subset J^1 M$ be a Legendrian. Then for all sufficiently small $\epsilon$ and any $0 < \delta < \epsilon$, the composition of the following maps is an equivalence:
    $$\Sh_{\Lambda \times \Lambda_D}^+(M \times (-\infty, \epsilon))^* \to \Sh^+_{\Lambda_{\pm \delta}}(M \times \delta)^* \to \mu \Sh^+_\Lambda(M)^*.$$
    Here, the first map is the restriction to the slice over $\delta \in (0, \epsilon)$, and $\Lambda_{\pm \delta}$ are the two copies of $\Lambda$ in the symplectic reduction of $\Lambda \times \Lambda_D$ along this slice. The second map is the microlocalization along one of these copies.
\end{theorem}

\subsection{Legendrian submanifolds without Reeb chords}
At this point, we note the following application, which does not require
the results of the next section.

\begin{thm}\label{thm:nochords}
Let $\Lambda\subset J^1 M$ be a closed Legendrian which is contact isotopic to a Legendrian submanifold which projects injectively to $T^* M$.
The maps $\GF_\Lambda(M)^*\to \mu \GF_\Lambda(M)^*$ and $\Sh^+_\Lambda(M)^*\to \mu\Sh^+_\Lambda(M)^*$ are homotopy equivalences.
\end{thm}
\begin{proof}
The proof is identical in the generating function or sheaf setting, we write it for generating functions.
Both terms are invariant under Legendrian isotopy so we may assume that $\Lambda$ projects injectively to $T^*M$.
The family $(\Lambda\times \Lambda_D)_t$ for $t>0$ is then a Legendrian isotopy. Hence the homotopy lifting property
for generating functions (Theorem~\ref{thm:spacelevelchekanov}) implies that the restriction map
$$\GF_{\Lambda\times \Lambda_D}(M \times \R)^* \to \GF^{\leq 0}_{\Lambda\times \Lambda_D}(M\times \R)^* :=\colim_{\epsilon >0} \GF_{\Lambda\times \Lambda_D}(M\times (-\infty,\epsilon))^*$$
is a homotopy equivalence. Hence the following is a composition of homotopy equivalences, in view of Theorem~\ref{thm:GF_RD}, Theorem~\ref{thm:muGF_RD},
Proposition~\ref{prop:muleq0} and Theorem~\ref{thm:GFdoubling}:
\[\GF_\Lambda(M)^* \xrightarrow{\DD} \GF_{\Lambda\times \Lambda_D}(M\times \R)^* \to \GF^{\leq 0}_{\Lambda\times \Lambda_D}(M\times \R)^* \to \mu \GF_{\Lambda\times \Lambda_D}^{\leq 0}(M\times \R)^* \xrightarrow{\RR} \mu \GF_\Lambda(M)^*.\]
\end{proof}
\begin{rem}
Theorem~\ref{thm:nochords} implies our main result Theorem~\ref{thm:main} in this particular case.
\end{rem}
\begin{rem}
Combined with the semi-local classification of generating functions from Giroux-Latour (Theorem~\ref{thm:spacelevelgiroux-latour}), we conclude that $\GF_\Lambda(M)^*$
is either empty or homotopy equivalent to $\Map(\Lambda,\Z\times BO)$.
It is non empty if the stable Gauss map of $\Lambda$ is homotopically trivial (which is predicted by
the nearby Lagrangian conjecture, and proved e.g. for homotopy spheres in \cite{ACGK}).

It is known that the stable Gauss map composed with $J:U/O\simeq B(\Z\times BO)\to B(\Z\times BG)$ is trivial
for such Legendrians (see \cite{Jin} or \cite{AACK}). Hence $\mu\Sh^+_\Lambda(M)^*$ is non-empty, and homotopically equivalent to $\Map(\Lambda,\Z\times BG)$.
Hence we conclude that $\Sh^+_\Lambda(M)^*\simeq \Map(M,\Z\times BG)$ (a result also proved in \cite{Jin-Treumann}).
\end{rem}

\begin{rem}
In \cite{ACGK}, it is shown that Legendrian submanifolds as in Theorem~\ref{thm:nochords} always admit
a (tame) \emph{twisted} generating function. We believe the above proof generalizes to the twisted case.
\end{rem}

\begin{rem}
Theorem~\ref{thm:nochords} applies to the case where $\Lambda$ is contact isotopic to the zero-section and thus provides
another proof of Viterbo-Théret's uniqueness theorem.
\end{rem}

\section{The crossing problem}\label{sec:crossing}
\subsection{The main statement}
For $0\leq k \leq m$,  consider the following functions on $\R^m$:
\[z_{k, m}^+(x)=1+x_1^2+\dots+x_{m-k}^2-x_{m-k+1}^2-\dots-x_m^2\]
and
\[z_{k,m}^-(x)=- z_{k, m}^+(x).\]
Note that $z_{k,m}^+-z_{k,m}^-$ has a single critical point, which is non-degenerate of index $k$.
We write $\Lambda_{k,m}\subset J^1(\R^m)$ for the Legendrian given by the union of the jets of the above functions (its front projection is the union of their graphs).
We denote the positive half line as $\R_+ := [0, \infty)$, and write $\Lambda_{k+1,m+1}^+$ for restriction of $\Lambda_{k+1,m+1}$ to the half space $J^1(\R^m \times \R_+)$.
Note that slicing $\Lambda_{k+1,m+1}^+$ by the $x_{m+1}$-coordinate defines a Legendrian regular homotopy (i.e. a path of Legendrian immersions) with a single crossing (self-intersection)
corresponding to the single Reeb chord of $\Lambda_{k,m}$.

Restriction at $x_{m+1}=0$ defines maps
\[\GF_{\Lambda_{k+1,m+1}^+}(\R^{m+1}_+)\to \GF_{\Lambda_{k,m}}(\R^m) \quad\text{ and }\quad \Sh^+_{\Lambda_{k+1,m+1}^+}(\R^{m+1}_+)\to \Sh^+_{\Lambda_{k,m}}(\R^m)\]
which are compatible with the map $\s$ from generating functions to sheaves. The crossing problem refers
to the problem of extending an object (generating function or sheaf) past the crossing, i.e. from $\R^m$ to $\R^{m+1}_+$.
We will define obstructions valued in stable homotopy groups of spheres, the main observation is that
these obstructions are precisely the same on both side (if sheaves are considered over $\S$). In other words we have the following result which is the main goal of this section:

\begin{thm}\label{thm:crossing-gf=sh}
The commutative diagram

\begin{equation}\label{model-square}
\begin{tikzcd}
\GF_{L_{k+1,m+1}^+}(\R^m\times \R_+)^* \arrow[d]\arrow[r] & \arrow[d] \GF_{\Lambda_{k,m}}(\R^m)^* \\
\Sh^+_{L_{k+1,m+1}^+}(\R^m\times \R_+;\S)^* \arrow[r] & \Sh_{\Lambda_{k,m}}^+(\R^m;\S)^*.
\end{tikzcd}
\end{equation}
is a homotopy pullback square. The fibers of the horizontal maps over an object giving Maslov indices $j$ and $i$
to $z^+_{k,m}$ and $z^-_{k,m}$ respectively is either empty (if some obstruction in the stable homotopy group $\pi_{k+j-i-1}^S$ is non-vanishing),
or homotopy equivalent to
$\Omega^{k+j-i+\infty}S^{\infty}$ (if the obstruction vanishes).
\end{thm}

For $f_0 \in \GF_{L_{k,m}}(\R^m)^*$ and $F_0\in\Sh^+_{L_{k,m}}(\R^m)^*$, let us write $\E(f_0)$ (respectively $\E(F_0)$)
for the fiber of the top (resp. bottom) horizontal map in \eqref{model-square}. Theorem~\ref{thm:crossing-gf=sh}
implies that the map
\begin{equation}
\E(f_0)\to \E(\s(f_0))
\end{equation}
is a homotopy equivalence; in particular,
one is non-empty if and only the other one is non-empty.

\begin{example}
In the case where $k=0$ and $j=i$, Theorem~\ref{thm:crossing-gf=sh} says that the extension is always possible since $\pi^S_{k+j-i-1}= \pi^S_{-1}=0$
and the space of extensions (fiber of horizontal maps in \eqref{model-square}) is
$\Omega^\infty S^\infty$
whose $\pi_0$ is $\Z$: this corresponds to an algebraic count of handleslides.
\end{example}

\begin{example}
When $k+j\leq i$, the obstruction lies in
$\pi^S_{\le -1} = 0$, so the extension is always possible.  This recovers the well-known fact that we can always reorder critical points according to their index (and even
in $k$-parametric families if the index difference is at least $k$).
\end{example}

The theorem is a model of the singular moment in a Legendrian regular homotopy.  In fact, when combined with the homotopy lifting property for generating functions and sheaves, this model case allows us to deduce the corresponding result
for any Legendrian regular homotopy:

\begin{cor}\label{cor:regularhtpy}
Let $\Lambda\subset J^1(M\times [0,1])$ be a compact Legendrian cobordism that is the trace of a Legendrian regular
homotopy $(\Lambda_t)_{t\in [0,1]}$ with $\Lambda_0$ and $\Lambda_1$ embedded. Then restriction at $t=0$ defines a commutative diagram

\begin{equation}
\begin{tikzcd}
\GF_\Lambda(M\times[0,1])^* \arrow[d]\arrow[r] & \arrow[d] \GF_{\Lambda_0}(M)^* \\
\Sh^+_\Lambda(M\times[0,1];\S)^* \arrow[r] & \Sh^+_{\Lambda_0}(M;\S)^*.
\end{tikzcd}
\end{equation}
which is a homotopy pullback square.
\end{cor}
\begin{proof}
We can perturb the Legendrian regular homotopy relative to $t=0$ and $t=1$ so that it has finitely
many transverse self-crossings arising at distinct parameters $0<t_1<\dots<t_l<1$. Moreover we can assume
that at each self-intersection, the crossing arises along two branches of $\Lambda_t$ which are smooth graphs
and no other branches of the front of $\Lambda_t$ passes through the crossing point. Hence each crossing of $\Lambda_{t_i}$
for $i=i,\dots,l$ can be modeled after the regular homotopy $(\Lambda_{k+1,m+1}^+)_{x_{m+1}}$ for some
integers $k$. Away from a small ball around the crossing point, $(\Lambda_t)_{t\in[0,1]}$ is a Legendrian isotopy so
Theorem~\ref{thm:spacelevelchekanov} and Theorem~\ref{thm:spacelevelGKS} apply. Hence, for small $\epsilon >0$, the fiber of
$\GF_\Lambda(M\times[0,t_{i+1}+\epsilon]) \to \GF_\Lambda(M\times [0,t_{i}+\epsilon])$ is isomorphic to
the corresponding fiber in the model case of $\Lambda_{k+1,m+1}^+$, and the result follows by applying
Theorem~\ref{thm:crossing-gf=sh} several times.
\end{proof}

The remainder of this section gives the proof of Theorem \ref{thm:crossing-gf=sh}.

\subsection{Reduction to surjectivity on $\pi_0$}

Theorem~\ref{thm:crossing-gf=sh} asserts that for any $f_0\in \GF_{\Lambda_{k,m}}(\R^m)^*$
the map $\E(f_0)\to \E(\s(f_0))$ is a homotopy equivalence. Employing the symmetries of the model Legendrian
submanifolds $\Lambda_{k,m}$, we show here that it is enough to prove surjectivity on $\pi_0$ for all $k,m$.

Given $f\in \E(f_0)$ and $l\in \N$, we define
$r^lf \in \GF_{\Lambda_{k+1+l,m+1+l}}(\R^{m+1+l})^*$ by the formula
\[r^l f(x,y;u)=f(x_1,\dots,x_m,\sqrt{x_{m+1}^2+|y|^2};u)\]
which geometrically means that we implant the function $f$ on each ray in $\R^{1+l}$ (with coordinates $(x_{m+1},y)$). For instance for $l=0$,
this is just the function $f$ mirrored along the hyperplane $x_{m+1}=0$.
There is a small smoothness issue along the origin in $\R^{l+1}$ which we will ignore (we could deal with it by requiring elements of $\E(f_0)$
to conform to a special model near $x_{m+1}=0$).
Similarly for $F_0\in \Sh^+_{\Lambda_{k,m}}(\R^m)^*$ and $F\in \E(\F_0)$, we define $r^l F \in \Sh^+_{\Lambda_{k+1+l,m+1+l}}(\R^{m+1+l})^*$ as

\begin{equation}
r^l F=p^* F
\end{equation}
where $p: \R^{m+1+l}\times \R\to \R^{m+1}\times \R$ is given by
\[p(x_1,\dots,x_m,x_{m+1},y,u)=(x_1,\dots,x_m,\sqrt{x_{m+1}^2+|y|^2},u).\]

There are homotopy equivalences
\begin{align}
\E(r^l f)\to \Omega^l \E(f_0)\\
\E(r^l F)\to \Omega^l \E(F_0)
\end{align}
where the based loop space $\Omega^l$ can be seen explicitly as space of maps $(D^l,\del D^l)\to (\E(f_0),f)$
where $D^l$ is the disk $\{(x_{m+1},\dots,x_{m+2+l}), x_{m+1}^2+\dots+x_{m+2+l}^2=1, x_{m+2+l}\geq 0\}$).
We obtain for each $l\in \N$ a commutative square
\begin{equation}
\begin{tikzcd}
\E(r^l f)\arrow[r,"\sim"] \arrow[d] & \Omega^l \E(f_0) \arrow[d] \\
\E(r^l \s(f))\arrow[r,"\sim"]&\Omega^l \E(\s(f_0)).
\end{tikzcd}
\end{equation}

And so passing to $\pi_0$;
\begin{equation}
\begin{tikzcd}
\pi_0\E(r^l f)\arrow[r,"\sim"] \arrow[d] & \pi_l \E(f_0) \arrow[d] \\
\pi_0\E(r^l \s(f))\arrow[r,"\sim"]&\pi_l \E(\s(f_0)).
\end{tikzcd}
\end{equation}

Hence Theorem~\ref{thm:crossing-gf=sh} follows if we can show that $\E(f_0)\to \E(\s(f_0))$ is an isomorphism
on $\pi_0$ (for any $k,m$ and $f_0\in  \GF_{\Lambda_{k,m}}(\R^m)^*$). In fact, injectivity on $\pi_0$ also follows
from surjectivity on $\pi_0$ using the following small variation
on the above construction.

Given $f_0\in \GF_{\Lambda_{k,m}}(\R^m)^*$,  $f_-, f_+ \in \E(f_0)$, we define $r(f_-,f_+)\in \GF_{\Lambda_{k+1,m+1}}^*$
by the formula
\[r(f_-,f_+)(x_1,\dots,x_m,x_{m+1};u)=\begin{cases}f_+(x_1,\dots,x_m,x_{m+1};u) \text{ if } x_{m+1}\geq 0 \\
f_-(x_1,\dots,x_m,-x_{m+1};u) \text{ if }x_{m+1}\leq 0.\end{cases}\]
We proceed similarly with sheaves to define $r(F_-,F_+)$ for $F_0 \in \Sh^+_{\Lambda_{k,m}}(\R^m)^*$ and $F_-,F_+\in \E(F_0)$.
Now $\E(r(f_-,f_+))$ is non-empty if and only if $[f_-]=[f_+]\in \pi_0 \E(f_0)$ and similarly for sheaves.
So if $\s(f_-)=\s(f_+)\in \pi_0\E(\s(f_0))$ and $\pi_0\E(r(f_-,f_+))\to \pi_0\E(r(\s(f_-),\s(f_+)))$ is surjective,
we conclude that $\E(r(f_-,f_+))$ is non-empty and thus $|f_-]=|f_+]\in \pi_0 \E(f_0)$.
So we obtain injectivity of $\pi_0\E(f_0)\to \pi_0\E(\s(f_0))$ in this way.

\subsection{The crossing obstruction}

Let us fix functions $a,b,c : \R^m\to \R$ defined by
\[a(x)=-\chi(\|x\|),\quad  b(x)=0, \quad c(x)=\chi(\|x\|),\]
where $\chi:[0,\infty)\to [0,2]$ is non-increasing, compactly supported and satisfies $\chi(r)=2$ for $r\leq 1$.
The space of such functions is convex and hence the precise choice of $\chi$ will be irrelevant. We depict the graphs of these functions along with the $Z_{\pm}$ in Figure \ref{abc figure}.

\begin{figure}
    \centering
\includegraphics{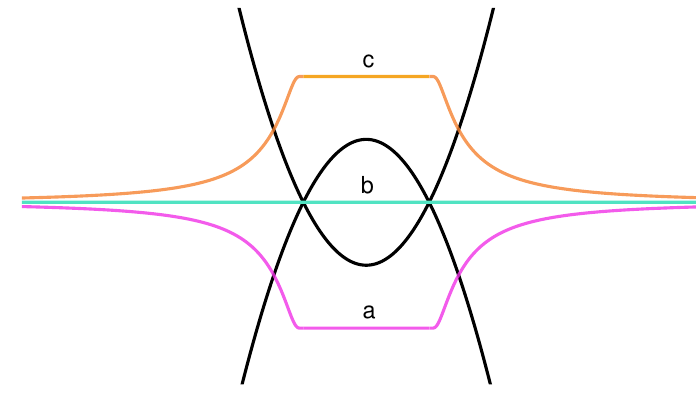}
    \caption{In color: the graphs of the functions $a, b, c$. In black: the graphs of $Z_+$ and $Z_-$.  Here $k=m=1$.}
    \label{abc figure}
\end{figure}

Given a generating function $f$ for $\Lambda_{k,m}$ we abbreviate
\[\{f \le a\} :=\{(x,v)\in\R^m\times \R^n,f(x,v)\leq a(x)\}\]
and similarly $\{f\leq b\}$, $\{f\leq c\}$.

\begin{dfn}\label{dfn:crossing-obstruction}
The \emph{crossing obstruction} of $f\in \GF_{\Lambda_{k,m}}(\R^m)^*$ is the map (of spectra)
\[\alpha(f):\{f \le c\} /\{ f \le b\} \to \Sigma(\{f \le b\} /\{ f \le a\})\]
from the Puppe sequence
\[\{f \le b\} /\{ f \le a\}\to \{f \le c\} /\{ f \le a\} \to \{f \le c\} /\{ f \le b\}\to \Sigma(\{f \le b\} /\{ f \le a\}) \to \cdots \]
\end{dfn}

If we add a dimension in the auxillary space and stabilize $f$ to $f+t^2$, then each term in the sequence is replaced
by a homotopy equivalent space. Alternatively, if we replace $f$ by $f-t^2$, then each term is suspended once
(all Morse indices are shifted up by one). In particular, as a map of spectra, $\alpha(f)$ is indeed defined
for $f\in\GF_{\Lambda_{k,m}}$ (recall the latter space is defined as a colimit under stabilization).

If the fiberwise critical points of $f$ have index $i$ and $j$ (corresponding to $z_{k,m}^-$ and $z_{k,m}^+$ respectively),
then there are homotopy equivalences
\[\{f\le c\}/\{f \le b\}\simeq S^{j+k}\quad \text{and}\quad \{f\le b\}/\{f \le a\}\simeq S^i.\]
Under these identifications (which for a fixed $f$ is determined up to homotopy by a choice of orientations of the descending manifolds
of the fiberwise critical points of $f$), we may view $\alpha(f)$ as a map
\[S^{j+k}\to S^{i+1}\]
or stably as an element of $\Omega^{j+k-i-1+\infty}S^\infty$.
Correspondingly the space $\NN(\alpha(f))$ of all stable null-homotopies of $\alpha(f)$ is either empty
or a torsor over $\Omega^{j+k-i+\infty}S^\infty$.

Now suppose given $f: \R^m\times  \R^n \to \R$
generating $\Lambda_{k,m}$ and an extension $f_+:\R^m\times \R_+\times \R^n \to \R$ of $f$ which generates $\Lambda_{k+1,m+1}^+$.
We preserve the notations $a, b, c, \alpha$ for their half-space counterparts; so we have a map
\[\alpha(f_+):\{f_+ \le c\}/\{f_+ \le b\} \to \Sigma(\{f_+ \le b\}/\{f_+ \le a\})\]
which fits into a commutative diagram

\begin{equation}\label{eq:nullhomotopy1}
\begin{tikzcd}
\{f \le c\}/\{f \le b\} \arrow[r,"\alpha(f)"]\arrow[d] &\Sigma(\{f \le b\}/\{f \le a\})\arrow[d] \\
\{f_+ \le c\}/\{f_+ \le b\} \arrow[r,"\alpha(f_+)"]& \Sigma(\{f_+ \le b\}/\{f_+ \le a\}).
\end{tikzcd}
\end{equation}

Now we observe that $\{f_+ \le c\}/\{f_+ \le b\}$ is a cone over $\{f \le c\}/\{f \le b\}$, while the inclusion $\{f \le b\}/\{f \le a\}\to \{f_+ \le b\}/\{f_+ \le a\}$
is a homotopy equivalence. Indeed, $\{f_+\leq c\}/\{f_+\leq b\}$ deformation retracts onto the union $\Delta^+$ of fiberwise stable disks
of the critical point corresponding to the branch $z_{k+1,m+1}^+$, and $\Delta^+$ itself a cone on $\Delta^+\cap\{x_{m+1}=0\}$ which is a deformation retract of $\{f\leq c\}/\{f\leq b\}$. Similarly the union $\Delta^-$ of the stable disks of the critical point corresponding to $z_{k+1,m+1}^-$ in $\{f_+\leq b\}/\{f_+\leq a\}$
is a cylinder on $\Delta^-\cap\{x_{m+1}=0\}$, and the latter is a deformation retract of $\{f\leq b\}/\{f\leq a\}$.

Hence the diagram \eqref{eq:nullhomotopy1} gives a null-homotopy of
the crossing obstruction $\alpha(f)$.

\begin{dfn}
The null-homotopy of $\alpha(f)$ induced by the extension $f_+$ is the map $\alpha(f_+)$.
This defines a map $\E(f_0)\to \NN(\alpha(f_0))$.
\end{dfn}

\begin{theorem} \label{thm:gfcrossing}
For any $f_0\in \GF_{\Lambda_{k,m}}(\R^m)^*$, the crossing obstruction map $\E(f_0)\to \NN(\alpha(f_0))$ is an isomorphism on $\pi_0$.
\end{theorem}
\begin{proof}
As argued in the previous subsection, it is enough to show surjectivity on $\pi_0$.

    The basic point is that
    we want to show that certain ascending and descending spheres can be disjoined after an isotopy.  Essentially by definition,
    $\alpha$ measures the obstruction for doing this up to stable homotopy, and the stable homotopy can be realized by an isotopy after sufficient stabilization of the generating function.  We now explain in more detail.

    It is instructive to consider first the case $m=0$.  Then $f$ is just a single Morse function, and we are discussing the problem of varying $f$ so that the critical values of two critical points cross.
    A necessary and sufficient condition for crossing the critical values is the existence
    of a descending sphere of $p$ which is disjoint from an ascending sphere of $q$ in some intermediate level set $\{f=b\}$, see e.g. \cite{Cerf-stratification, Hatcher-Wagoner}.
    Now the descending sphere $S^{j-1}\to \{f=b\}$ is homotopic to a sphere disjoint from the ascending sphere of $q$ if and only if the composition $S^{j-1}\to \{f=b\}\to \{f\leq b\}$ factors up to homotopy  through a map $S^{j-1}\to \{f \leq a\} \hookrightarrow \{f \leq b\}$.  By construction of $\alpha$,  this is true up to stable homotopy iff $\alpha(f) = 0$.  This stable homotopy can be realized by an actual homotopy if the pair $(\{f\leq b\},\{f\leq a\})$ and $\{f\leq a\}$ are sufficiently connected.   We claim that this can be achieved by sufficient negative stabilization of $f$.  Indeed, this is obvious for $(\{f\leq b\},\{f\leq a\})$, as the Morse indices are all shifted upwards.  As for  $\{f\leq a\}$, observe that the contractible space $\R^n$ is obtained from $\{f\leq a\}$ by attaching cells of index greater than say $l\geq 0$, so $\{f \leq a\}$ is  at least $(l-2)$-connected, and similarly $\{ f - t^2\leq a\}$ will be at least $(l-1)$-connected.\footnote{This is the only place in this article where we use that $f$ is defined on $\R^n$ as opposed to an arbitrary manifold.}
    Finally, a homotopy of submanifolds can be realized by an isotopy if the codimension is sufficiently large; we raise the codimenion by positive stabilization.
    All in all, vanishing of $\alpha$ implies we may cross the critical values after sufficient stabilization. Moreover by construction
the null-homotopy $\alpha(f_+)$ of $\alpha(f)$ induced by the constructed extension $f_+$ is homotopic to
the null-homotopy of $\alpha(f)$ that we have used to construct $f_+$.

    Let us consider the general case. Consider the descending sphere $\phi:S^{j+k-1}\to \{(x,v),f(x,v)=b(x)\}$.
    The intersection $S^{j+k-1} \cap \{x \in \partial \Delta^k\}$ is a $S^{k-1}$; on the complement of this sphere, the map $\phi$ fibers as
    $\phi_x:S^{j-1}\to \{v,f(x,v)=b(x)\}$.
    Similarly the ascending sphere of the lower critical point is described as a family $\psi_x:S^{n-i-1}\to \{v, f(x,v)=b(x)\}$. For $x\in \del \Delta\simeq S^{k-1}$, $\phi_x$ and $\psi_x$ are disjoint since the critical values are equal.
    The crossing is possible if and only if there is a disjunction isotopy, namely an isotopy of $\phi_x$ which makes it disjoint from $\psi_x$ and which is fixed for $x\in \del \Delta$ \cite[Theorem 4.11]{igusa88}.

    As before, $\alpha(f)$ is the obstruction to stable-homotopy-disjoining the ascending and descending spheres $\phi$ and $\psi$; after stabilizing, we may realize this by a homotopy, which, again after stabilizing, can be taken to fix
    the already-disjoint behavior over $\del \Delta^k$.

    Because of the fibration structure, the existence of such a disjoining homotopy implies the existence of a fiberwise such homotopy. Indeed, let $E = \{f \le a\} \setminus \psi(S^{n-i-1})$, and consider the fibration $\pi: E \to M = \R^m$.  Let $\phi(t): S^{j+k+1} \to E$ be the witness to the homotopy disjoining, with $\phi_0$ the original $\phi$.  But then in the complement of the $S^k$ (along which $\phi(t)$ is constant), the map $\pi \circ \phi(t)$ is a fibration, which we may therefore lift along $\phi$ to a fiberwise homotopy-disjoining $\phi'(t)$ (i.e. so that $\phi'_x(t)$ disjoins $\phi_x(t)$ from $\psi_x(t)$).

    Finally, as before, by stabilization we may raise the connectivity of the relevant ambient spaces and the codimension of the embedding, hence ensure that the  homotopy is realized by an isotopy.
\end{proof}

\subsection{Crossing for sheaves}\label{subsec:crossingsheaves}

Unlike in the remainder of this article, in this subsection we switch to cohomological degree conventions, because this is more common in the sheaf literature and also because in this subsection we will make no choice or mention of the category of coefficients $\mathcal{C}$.  Here we will officially work with and use notation for $\Sh(M \times \R)$; the results of course restrict to $\Sh^+(M)$ (by viewing it as a full subcategory).

\subsubsection{Crossing of microsupport lines}

We consider sheaves on $\R^2_{t, z}$ whose microsupport away from the zero section is contained in the upward conormals to the lines $z=t$ and $z = -t$.
Let us denote the microsupport condition as $ \times$, and correspondingly the category of sheaves by $Sh_{\times}({\R^2_{t, z}})$.

Given such a sheaf, one can show by considering non-characteristic propagation that there are maps between the stalks in the various regions, as indicated in Figure \ref{crossing diagram}.  (The $t$ axis is horizontal, and the $z$ axis is vertical.)
In fact, the  category $Sh_{\times}({\R^2_{t, z}})$  is equivalent to the category of diagrams as indicated which are cartesian/co-cartesian \cite{STZ}.

\begin{figure}
    \centering
\includegraphics{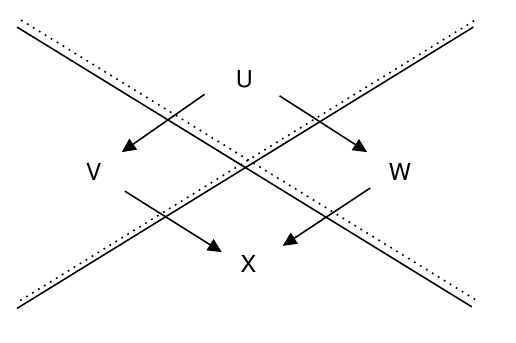}
    \caption{Sheaves on a crossing.}
    \label{crossing diagram}
\end{figure}

There is a natural restriction map
$Sh_{\times}({\R^2_{t, z}}) \to Sh_{\times}(\R_{t < 0} \times \R_z) \cong Sh_{\times}(\{-1\} \times \R_z)$.  We are interested in the image and the fibers.  By the above characterization, this is equivalent to asking about the image and fibers of the restriction on diagram categories
$$[X = V \oplus_U W] \mapsto [U \to V \to X].$$

These may be characterized as follows:
\begin{lem} \label{0d sheaf split}
    Isomorphism classes of extensions of a diagram $U \to V \to X$ to cartesian diagrams of the form $X = V \oplus_U W$ are in bijection with isomorphism classes of splittings of the exact triangle
    \begin{equation} \label{splitme}
        Cone(V \to X) \to Cone(U \to X) \to Cone(V \to U) \xrightarrow{[1]}
    \end{equation}
\end{lem}
\begin{proof}
    The problem is unchanged upon replacing $U, V, W, X$ with their cones with respect to the given map from $U$.  Now the assertion is that
    splittings of $0 \to V \to X$ are in bijection with decompositions $X = V \oplus W$, which is tautologically true.
\end{proof}

Note that, in general,  the obstruction to the splitting of an exact triangle $A \to B \to C$ is the vanishing of the connecting map $C \to A[1]$, and the isomorphism classes of splittings are in bijection with homotopy classes of null-homotopies of $C \to A[1]$, hence form a torsor for $\Hom(C, A[2])$, etc.

Finally let us note that the situation discussed above
differs by a diffeomorphism from the extension problem from $\Lambda_{0,1}$ to $\Lambda_{1,2}^+$; said diffeomorphism induces an equivalence of sheaf categories.  We chose this alternate model solely to draw a simpler picture.

\subsubsection{General case}

\begin{lem} \label{sheaf precrossing characterization}
    Automorphism classes of objects in $\Sh_{\Lambda_{k,m}}(\R^m\times\R_z)$ are in bijection with automorphism classes of diagrams $U \to V \to W$ equipped with a locally constant family of splittings of this constant sequence over $\partial D^k$.
\end{lem}
\begin{proof}
    We write $\R^m = \R^m_{>} \sqcup \R^m_{=} \sqcup \R^m_{<}$ according as $z_+ > z_-$ or $z_+ = z_-$ or $z_+ < z_-$.  Observe that $\R^m_>$ retracts to the interior of $\Delta$, while $\R^m_=$ retracts to $\partial \Delta$.

    It is clear that $\Sh_{\Lambda_{k,m}}(\R^m_{>} \times \R_z)$ is equivalent to the category of sequences $U \to V \to W$.  Indeed, setup is diffeomorphic to the product of a contractible factor and a line with microsupport at the positive conormals to two points.

    Let us now consider $\Sh_{\Lambda_{k,m}}(Nbd(\overline{\R^m_{>}}) \times \R_z)$, i.e. how the category changes after
    including a neighborhood of the hypersurface $\R^m_{=}$.  Note this is a $\R^m_{=}$-family worth of the geometry discussed in Lemma \ref{0d sheaf split}.  So, the category of such sheaves is equivalent to the category of sequences $U \to V \to W$, along with a local system of splittings of said sequence along $\R^m_{=}$.  As we have already noted, the $\R^m_=$ retracts to $\partial \Delta$.

    Finally, restriction induces an equivalence $\Sh_{\Lambda_{k,m}}(\R^m \times \R_z) = \Sh_{\Lambda_{k,m}}(Nbd(\overline{\R^m_{>}}) \times \R_z)$.
\end{proof}

\begin{lem} \label{sheaf family splitting problem}
    Fix $F \in \Sh_{\Lambda_{k,m}}(\R^m \times \R_z)$ corresponding to $U \to V \to W$ with some fixed splitting $\sigma$ over $S^{k-1} = \partial D^k$.  Then
    isomorphism classes of extensions of $F$ to
    $\widetilde{F} \in \Sh_{\Lambda_{k+1,m+1}^+}(\R^{m+1}_+ \times \R_z)$ are in bijection with isomorphism classes of extensions of $\sigma$ over $D^k$.
\end{lem}
\begin{proof}
    Note that the front of $\Lambda_{k+1,m+1}^+$ is just the transverse intersection of two graphs along  a contractible locus, i.e. overall it is diffeomorphic to the product of $\Sh_{\times}({\R^2_{t, z}})$ with a trivial factor, hence determines the same category of sheaves.
\end{proof}

Let us now recall a general lemma about extending splittings of exact triangles

\begin{lem} \label{some homological algebra}
   Suppose in the following square, the rows and columns are exact triangles:
$$
 \begin{tikzcd}
    A \ar[r] \ar[d] & B \ar[r] \ar[d] & C \ar[d] \\
    A' \ar[r] \ar[d] & B' \ar[r, "y'"] \ar[d, "\pi"] & C' \ar[d] \\
    A'' \ar[r]  & B'' \ar[r]  & C''
\end{tikzcd}
$$
    Suppose given a splitting of the middle exact triangle, e.g. by a section $\gamma': C' \to B'$ of the map $y: B' \to C'$.  Then there is an obstruction $o(\gamma'): C \to A''$ which vanishes iff all three horizontal sequences split.
\end{lem}
\begin{proof}
    Let us attempt to define a map $\gamma: C \to B$.  So consider an element $c \in C$.  Take its image $c' \in C'$; we should then try and lift $\gamma'(c')$ from $B$ to $B'$.  The obstruction to doing this is $\pi(\gamma'(c))$.  By commutativity, the image of this element in $C''$ vanishes, hence it lifts canonically to some $o(c) \in A''$.

    This defines the desired map $o: C \to A''$, and similar considerations show that null-homotopies of this map correspond to splittings $\gamma: C \to B$, which then induce corresponding $\gamma'': C'' \to B''$ by functoriality of cones.  (We of course work in the setting of stable categories.)
\end{proof}

\begin{cor}
   Suppose given a map $A \to B$ and a family of splittings of said map over $\partial D^k$, inducing in particular $\gamma': \Gamma(\partial D^k, B) \to \Gamma(\partial D^k, A)$.  The obstruction to extending the splitting over $D^k$ is the obstruction map $o(\gamma'): Cone(A \to B)[1-k] \to A$ associated to the following square:
   $$
 \begin{tikzcd}
    A[1-k] \ar[r] \ar[d] & B[1-k] \ar[r] \ar[d] & Cone(A\to B)[1-k] \ar[d] \\
    \Gamma(S^{k-1}, A) \ar[r] \ar[d] & \Gamma(S^{k-1}, B) \ar[r, "y'"] \ar[d] & \Gamma(S^{k-1}, Cone(A \to B)) \ar[d] \\
     A \ar[r]  & B \ar[r]  & Cone(A \to B)
\end{tikzcd}
$$
\end{cor}
\begin{proof}
    Immediate from Lemma \ref{some homological algebra}.
\end{proof}

\begin{figure}
    \centering
    \includegraphics{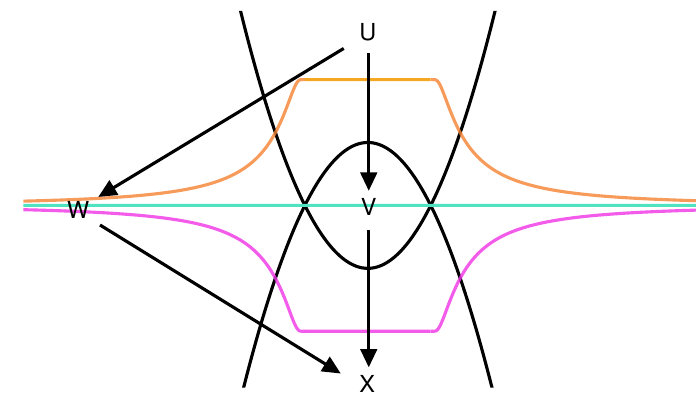}
    \caption{Schematic depiction of data associated to an object of $\Sh_{\Lambda_{k,m}}(\R^m \times \R_z)$.  Note that $U, V, X$ are all objects of the coefficient category, whereas $W$ is a local system of such objects on the region described (which is homotopy equivalent to $\partial \Delta^k$).}
    \label{precrossing figure}
\end{figure}

Consider now an object $F \in \Sh_{\Lambda_{k,m}}(\R^m  \times  \R_z)$. We wish to describe the obstruction to extending $F$ to some $\widetilde{F} \in \Sh_{\Lambda_{k+1,m+1}^+}(\R^{m+1}_+  \times  \R_z)$ directly in terms of the sections of $F$.  To this end, consider the regions depicted in Figure \ref{precrossing figure}.  We may consider the restriction maps
$\Gamma(\{z \le c\}, F) \to \Gamma(\{z \le b\}, F) \to \Gamma(\{z \le a\}, F)$.
From these we form the exact triangle on the opposite side of the octahedron:
$$\Gamma(\{z \le c\}, \{z \le b\}, F) \to \Gamma(\{z \le c\}, \{z \le a\}, F) \to \Gamma(\{z \le b\}, \{z \le a\}, F)$$
Here, for $S \subset T$ we write $\Gamma(T, S, F) := Cone(\Gamma(T, F)) \to \Gamma(S, F))$ for the relative sections.

\begin{theorem} \label{sheaf crossing obstruction}
    The space of extensions of an object $F \in \Sh_{\Lambda_{k,m}}(\R^m  \times  \R_z)$ to some $\widetilde{F} \in \Sh_{\Lambda_{k+1,m+1}^+}(\R^{m+1}_+  \times  \R_z)$
    is isomorphic to the space of null-homotopies of the connecting homomorphism
    $\alpha(F): \Gamma(\{z \le b\}, \{z \le a\}, F) \to \Gamma(\{z \le c\} \{z \le b\}, F)[1]$.
\end{theorem}
\begin{proof}
    We proceed with reference to Figure \ref{precrossing figure}.

    The sheaf $F$ admits a map from the constant sheaf with stalk $U$; replacing $F$ with the cone by this map preserves the hypothesis and conclusion.  Thus without loss of generality, we may assume $U = 0$.  Then, according to Lemma \ref{sheaf precrossing characterization}, the map $V \to X$ sits in the setting of Lemma \ref{some homological algebra}, and from that Lemma we obtain a map $Cone(V \to X) \to V[1-k]$, characterizing extensions of splittings, and hence, by Lemma \ref{sheaf family splitting problem}, extensions of $F$.  It remains to show that this map agrees with the connecting homomorphism $\alpha(F)$.

    We compute by \u Cech cover (returning for the moment to the case when $U$ need not be zero):
    \begin{eqnarray*}
    \Gamma(\{f \le c\}, F) & = & U \\
    \Gamma(\{f \le b\}, F) & = & Cone(\Gamma(\partial D^k, U) \oplus V \to  \Gamma(\partial D^k, V))[1] \\
    \Gamma(\{f \le a\}, F) & = & Cone(\Gamma(\partial D^k, W) \oplus X \to \Gamma(\partial D^k, X))[1]
    \end{eqnarray*}
    (For $\{f \le b\}$ it is helpful to first perturb $b = 0$ to $b = \epsilon$; one can easily see this does not change sections by noncharacteristic propagation.)

    We return to the $U=0$ setting.  Evidently
    $\Gamma(\{z \le c\}, \{z \le b\}, F) = V[1-k]$.  Since
    $\Gamma(\partial D^k, V) \oplus \Gamma(\partial D^k, W) \to \Gamma(\partial D^k, X)$ is null-homotopic by the property of sheaves microsupported in the crossing, we find  $\Gamma(\{z \le b\}, \{z \le a\}, F) = Cone(V \to X)[1]$. Finally, the diagram chase to define the connecting map
    $\alpha(F)$ is  the same as that we used to construct the splitting extension obstruction map in Lemma \ref{some homological algebra}, so these maps agree.
\end{proof}

\subsection{Proof of Theorem \ref{thm:crossing-gf=sh}}

Theorems \ref{thm:gfcrossing} and \ref{sheaf crossing obstruction} show that, for each $f_0\in \GF_{L_{k,m}}(\R^m)$, both spaces of extensions $\E(f_0)$ and $\E(\s(f_0))$ are homotopy equivalent to spaces of null-homotopies of certain maps $\alpha(f)$ and $\alpha(\s(f))$. By Equation \eqref{relative sections of Sf}, the corresponding maps are interchanged by $\Hom(\,\cdot\,, \S)$. This duality is invertible since everything in sight is a finite CW complex, hence induces a bijection on isomorphism classes. $\square$

\section{Proofs of the main results}\label{sec:mainproofs}
\begin{proof}[Proof of Theorem~\ref{thm:main}]
Fix a closed Legendrian $\Lambda\subset J^1(M)$. Recall from Section~\ref{sec:doubling} that $\Lambda\times \Lambda_D$ is a Legendrian submanifold of $J^1(M\times \R)$
whose slices for fixed $t\in \R$ are empty for $t<0$ and equal to $\Lambda_t=(\Lambda\pm 2t^{3/2})$ for $t>0$. In particular, for small $\epsilon >0$ and large $C>0$,
$(\Lambda_t)_{0<t\leq \epsilon}$ and $(\Lambda_t)_{t\geq C}$ are Legendrian isotopies while $(\Lambda_t)_{\epsilon\leq t\leq C}$ is a Legendrian regular homotopy.
For short we denote $\GF^{\leq t_0}=\colim_{\eta >0} \GF_{\Lambda\times \Lambda_D}(M\times (-\infty,t_0+\eta))$ and similarly for $\Sh$.
In particular, Corollary~\ref{cor:regularhtpy} and Theorem~\ref{thm:spacelevelchekanov} applies to show that the diagram
\[\begin{tikzcd}
 \GF^{\leq +\infty}\simeq \GF^{\leq C}\arrow[r]\arrow[d] & \GF^{\leq 0}  \arrow[d]\\
\Sh^{\leq +\infty}\simeq \Sh^{\leq C} \arrow[r] & \Sh^{\leq 0}.
\end{tikzcd}\]
is a homotopy pullback square.
The spaces on the left-hand side are equivalent to $\GF_\Lambda$ and $\Sh^+_\Lambda$ (see Theorem~\ref{thm:GF_RD} and Theorem~\ref{thm:Sh_RD}), while on the right-hand side
we have $\mu\GF_\Lambda$ and $\mu\Sh^+_\Lambda$ (see Theorem~\ref{thm:muGF_RD}, Theorem~\ref{thm:GFdoubling} and Theorem~\ref{thm:Shdoubling}).
\end{proof}

\begin{proof}[Proof of Theorem~\ref{thm: Z to S}]
In the regular homotopy $(\Lambda_t)_{t>0}$ obtained from the doubling trick (see Section~\ref{sec:doubling}),
the crossings are in one-to-one correspondence with the Reeb chords of the original Legendrian submanifold $\Lambda$.
Moreover for a Reeb chord of degree $k$, the corresponding crossing obstruction $\alpha$ (see Section \ref{sec:crossing})
is valued in $\pi_{k-1}^S$ and the space of possible crossings is either empty (if $\alpha\neq 0$) or has $\pi_0$ equal to $\pi_k^S$.

If $\Lambda$ has no Reeb chords of degree larger than $1$, the crossing obstructions are valued in $\pi_l^S$ with $l\leq 0$
which is the trivial group or $\Z$ for $l=0$. In the case $l=0$, the obstruction is detected by homology with $\Z$-coefficients
and hence, by Theorem \ref{sheaf crossing obstruction}, always vanishes if a sheaf over $\Z$ is given. Hence we can always lift from $\Z$ to $\S$ in this case, i.e.
the map \eqref{eq:mainfiberedZtoS} is $0$-connected.

If further $\Lambda$ has no Reeb chords of positive degree, then the space of extensions past each crossing is connected
or has $\pi_0$ equal to $\pi_0^S$ which again is detected by homology with $\Z$-coefficients and hence the lift from $\Z$ to $\S$ is unique.
More precisely the space of extensions in the case of integer coefficients has $\pi_0$ isomorphic to $\Z$
and all higher homotopy groups are zero (obstructions are valued in the trivial group $\pi_k \{\pm 1\}$, $k>0$)
this translates into $1$-connectedness of the map \eqref{eq:mainfiberedZtoS}.
Similarly if all Reeb chords have degree at most $1-l$, then the map \eqref{eq:mainfiberedZtoS} is $l$-connected.
\end{proof}

\section{Examples and applications}\label{sec: examples}
\subsection{Generalities and $n$-dimensional Legendrian unknots}

Let $\Lambda\subset J^1 M$ be any closed Legendrian submanifold. Recall the stable Gauss map $\gamma_\Lambda : \Lambda \to U/O\simeq B(\Z\times BO)$
and its J-homomorphism image $J\circ\gamma : \Lambda\to B(\Z\times BG)$.
The spaces of microlocal objects: $\mu\GF_\Lambda(M)^*$ is either empty (if $\gamma$ is non trivial) or a torsor over $\Map(\Lambda, \Z\times BO)$
and $\mu \Sh^+_\Lambda(M)^*$ is either empty (if $J\circ\gamma$ is non trivial)
or a torsor over $\Map(\Lambda,\Z\times BG)$  (see Theorem~\ref{thm:spacelevelgiroux-latour} and Theorem~\ref{thm:spaceleveljin}).
Assume $\gamma_\Lambda$ is trivial, then according to Theorem~\ref{thm:main} the spaces $\GF_\Lambda(M)^*$ and $\Sh^+_\Lambda(M)^*$ thus sit in the
pullback square
\begin{equation}\label{eq:maintrivialgauss}
\begin{tikzcd}
\GF_\Lambda(M)^* \arrow[r]\arrow[d]&\Map(\Lambda,\Z\times BO)\arrow[d]\\
\Sh^+_\Lambda(M)^* \arrow[r]& \Map(\Lambda,\Z\times BG).
\end{tikzcd}
\end{equation}
The fiber of the right vertical map (and thus also of the left vertical map) can be identified with the space $\Map(\Lambda, G/O)$ (recall $G/O$
is the homotopy fiber of the map $BO\to BG$ and is a classifying space for stable vector bundles equipped with a spherical trivialization).

\begin{example}[Lifting for Legendrian spheres]\label{ex:liftingspheres}
Suppose $\Lambda$ is an $n$-sphere (or even a homotopy $n$-sphere) with trivial stable Gauss map.
Then \eqref{eq:maintrivialgauss} provides an exact sequence:
\[\pi_n G/O \to \pi_0 \GF_\Lambda(M)^* \to \pi_0 \Sh^+_\Lambda(M)^* \to \pi_{n-1} G/O.\]
From the following table
\begin{equation}\label{eq:pi_nG/O}
    \begin{array}{l|cccccccccccc}
        n & 0 & 1 & 2 & 3 & 4 & 5 & 6 & 7 & 8 & 9 & 10\\ \hline
        \pi_n(G/O) & 0 & 0 & \Z/2 & 0 & \Z & 0 & \Z/2 & 0 & \Z\oplus \Z/2 &(\Z/2)^3  & \Z/2 \oplus \Z/3 \\

    \end{array}
\end{equation}
we learn for instance that for $n=1$ any sheaf over $\S$ lifts to a generating function and uniquely so, for $n=2$ any sheaf lifts and in at most two different ways
and for $n=3$ there are potentially sheaves that do not lift to generating functions (see Example~\ref{ex:3-spherenotlifting} below for an example of such a sheaf in a different context)
but
if they do, the lift is unique, etc.
\end{example}

\begin{example}[Legendrian unknot]
Let $\Lambda\subset J^1 M$ be (Legendrian isotopic to) the Legendrian unknot $S^n$ (whose front in a Darboux chart $\R^{2n+1}$
is a flying saucer) with $n\geq 1$.
In this case an element $F\in\Sh_\Lambda^+(M)^*$ is determined (up to contractible choice) by its stalk at any point inside
the flying saucer. Hence $\Sh_\Lambda^+(M)^* \simeq \Pic(\S)\simeq\Z\times BG$ and the map $\Sh_\Lambda^+(M)^* \to \mu\Sh_\Lambda^+(M)^*$
corresponds to the inclusion of constant maps $\Z\times BG \to \Map(\Lambda,\Z\times BG)$.
Identifying $\Lambda$ with $S^n$ and using the loop space structure of $\Z\times BG$ we find a splitting

\[\Map(\Lambda,\Z\times BG)\simeq \Z\times BG \times \Omega^n (BG)\]
and similarly for $BO$ (here $\Omega^n$ denotes the iterated based loop space). Hence, according to Theorem~\ref{thm:main}, $\GF_\Lambda^*$ is the homotopy pullback of the following square
\begin{equation}
\begin{tikzcd}
\GF_\Lambda^* \arrow[r]\arrow[d]&\Z\times BO \times \Omega^n (BO)\arrow[d]\\
\Z \times BG\arrow[r]& \Z\times BG \times \Omega^n (BG),
\end{tikzcd}
\end{equation}
from which we obtain a homotopy equivalence
\begin{equation}\label{eq:unknot}
\GF_\Lambda^*\simeq \Z\times BO\times \Omega^n (G/O)
\end{equation}
and in particular the classification of generating functions up to equivalence (see Proposition~\ref{prop:pi_0=equiv}):
\[\pi_0\GF_\Lambda^* \simeq \Z\times \pi_n (G/O).\]
The $\Z$-factor corresponds to a shift of grading (stabilization by a non-degenerate quadratic form
of non-zero signature). For $n=1$ we recover the fact that the Legendrian unknot has a unique generating function (up to shift),
a result of Jordan and Traynor \cite{Jordan-Traynor}. For $n=2$ we find two inequivalent generating
functions for the Legendrian unknot, and more examples can be read from \eqref{eq:pi_nG/O}.

For $n=0$, $\Lambda$ consists of two points $a<b$ in $\R=J^1(*)$. Let us denote $\GF_{\Lambda,0}(*)^*$ and $\Sh_{\Lambda,0}^+(*)^*$
the union of connected components corresponding to sheaves with stalk isomorphic to zero at any $z>b$.
In terms of generating functions, this corresponds to the space of functions $f$ on $\R^n$ with two critical points $p$ and $q$,
say with $f(p)=a$ and $f(q)=b$, that are in ``algebraic cancellation position''.
From our main result, we get in this case
\begin{equation}\label{eq:0-unknot}
\GF_{\Lambda,0}^* \simeq \Z\times BO\times G/O
\end{equation}
which can be understood as follows. For $f\in \GF_{\Lambda,0}^*$, we consider the Hessian of $f$ at $p$ and $q$ which define
(up to contractible choice) eigenspaces $E^+_p$, $E^-_p$ and $E^+_q$, $E^-_q$. The dimension of $E^-_p$
is one less than that of $E^-_q$ (because $p$ and $q$ are in algebraic cancellation position)
and there is a natural isomorphism $\theta:E^-_p\simeq E^-_q$ as stable spherical fibrations, given concretely
by our crossing obstruction map $\alpha(f)$. In the splitting \eqref{eq:0-unknot}, the $\Z\times BO$ factor
corresponds to the virtual vector bundle $E^+_p-E^-_p$ and the $G/O$ factor corresponds to
the virtual vector bundle $E^-_p-E^-_q$ with the above spherical trivialization $\theta$.

Similar considerations for $n\geq 1$ can be used to explain the homotopy equivalence \eqref{eq:unknot}. The difference being
that near the cusp locus of $\Lambda$, the spherical trivialization $\theta$ is lifted to a trivialization
as a vector bundle, since $p$ and $q$ are in ``geometric cancellation position'' there.
\end{example}

\begin{example}\label{ex:3-spherenotlifting}
Let $\Lambda \subset J^1(\R^3)$ be the $2$-copy of the three dimensional Legendrian unknot. Since $\pi_3 (\Z\times BG)\simeq \pi_2^S\simeq \Z/2$,
there are two sheaves $F_1$, $F_2$ over $\S$ for $\Lambda$ (up to shift) coming from microlocal sheaves on the unknot (see Theorem~\ref{thm:Shdoubling}).
However, by Theorem \ref{thm:spacelevelgiroux-latour}, microlocal generating functions are classified by $\pi_3(\Z\times BO)\simeq \pi_2 O=0$, so at most one of these sheaves can possibly lift; in fact, one does, since the $2$-copy necessarily admits a generating function.
\end{example}

\begin{remark}\label{rem:cerfuniquenessofdeath}
The uniqueness of generating functions (up to shift) for the $1$-dimensional Legendrian unknot can be thought of
as a poorman's version of Cerf's ``uniqueness of death'' lemma which states this result unstably with precise dimension
and connectivity conditions. In particular it implies that, for a $1$-dimensional Legendrian link in $J^1(\R)$, if we try
to construct a generating function progressively ``from left to right'', the extension over each right cusp, if possible,
is unique (up to homotopy).
\end{remark}

\begin{remark}\label{rem:hatcherwaldhausen}
The space $\GF(\ast)$ has many components, a component of which, denoted $\HH_\infty$, consists of functions $f$ for which $\{-c\leq f\leq c\}$
is a (necessarily trivial) stable $h$-cobordism (for large $c>0$). The map $\GF_{\Lambda,0}(*)^*\to\GF(\ast)$ takes values in $\HH_\infty$.
This is intimately related to the Hatcher-Waldhausen map $G/O\to \HH_\infty$ (see \cite{kragh_HW_2018}).
\end{remark}

\subsection{Dimension one}

We now focus on the case of a Legendrian link $\Lambda$ in $J^1(M)$ with $M=\R$ or $M=S^1$.
In this low dimension the map
\begin{equation}\mu\GF_\Lambda(M)^*\overset{\sim}{\longrightarrow}  \mu\Sh_\Lambda^+(M;\S)^*
\end{equation}
is $1$-connected (even an isomorphism on $\pi_1$). Our main result Theorem~\ref{thm:main} therefore implies that the map
\begin{equation}
\GF_\Lambda(M)^*\overset{\sim}{\longrightarrow} \Sh_\Lambda^+(M;\S)^*
\end{equation}
is $1$-connected, and in particular induces an isomorphism
\begin{equation}
\pi_0\GF_\Lambda(M)^*\overset{\sim}{\longrightarrow} \pi_0\Sh_\Lambda^+(M;\S)^*.
\end{equation}
The classification of generating functions up to equivalence is thus completely reduced to the classification
of sheaves over $\S$ in this case.

The map
\begin{equation}\mu\Sh_\Lambda^+(M;\S)^*\overset{\sim}{\longrightarrow}  \mu\Sh_\Lambda^+(M;\Z)^*
\end{equation}
is $1$-connected.
By Theorem~\ref{thm: Z to S}, the image of the map
\begin{align}
\pi_0\Sh_\Lambda^+(M;\S)^*\overset{\sim}{\longrightarrow}  \pi_0\Sh_\Lambda^+(M;\Z)^*
\end{align}
therefore
contains the subset of $\pi_0\Sh_\Lambda^+(M;\Z)$ which assigns a degree at most $1$ to each Reeb chord
(if $\Lambda$ is connected, this condition on the degree of Reeb chords is independent of the sheaf and depends only on $\Lambda$).
So any such sheaf over $\Z$ lifts to a sheaf over $\S$, and then further to a generating function.

In fact we can get a stronger result in the case $M=\R$ by slicing our Legendrian from left to right
instead of using the doubling trick. Indeed it is natural in this case to try and extend objects progressively from left to right past each bifurcation:
left cusps (births), double points (crossings) and right cusps (deaths).
By a Legendrian isotopy, we can assume that they appear in this order from left to right
and by applying Ng's resolution procedure (again a Legendrian isotopy), we can further make sure that the Reeb chords of $\Lambda$
are in bijection with crossings (and have degree $j-i$ if the crossing concerns branches of Maslov potential $j$ and $i$,
the branch of potential $j$ being above the other one before the crossing) and right cusps (which have degree $1$).
In this situation the assumption that all Reeb chords have degree at most $1$ is fulfilled if all crossings
concern branches of the same index. We say that a Maslov potential satisfying the latter condition is \emph{admissible}. This holds for instance for Legendrian braid closures (see e.g. \cite{Kalman} for a discussion of these Legendrians).
\begin{thm} \label{strong thm b}
Let $\Lambda$ be a Legendrian link in $J^1(\R)$ with a generic front projection and an admissible Maslov potential $\mu$.
Then the map $\pi_0\Sh_{\Lambda,\mu}^+(M;\S)^* \to \pi_0\Sh_{\Lambda,\mu}^+(M;\Z)^*$ is an isomorphism,
where the index $\mu$ means we restrict to objects inducing the Maslov potential $\mu$.
\end{thm}
\begin{proof}
We try to lift an element $F\in\Sh_{\Lambda,\mu}^+(M;\Z)$ to $\tilde{F}\in \Sh_{\Lambda,\mu}^+(M;\S)$ progressively
from left to right.

At each left cusp, the extension is always possible and unique up to a shift in grading (this is in fact a particular case
of Theorem~\ref{thm:Shdoubling}). At each right cusp, the extension is possible if the critical points about to cancel
are in cancellation position, a fact detected by homology with $\Z$-coefficients. Furthermore, if possible, the extension
past the right cusp is unique (see Remark~\ref{rem:cerfuniquenessofdeath}). Finally the problem of extension at each crossing has been analysed in Section~\ref{sec:crossing}, under our assumption on the Maslov potential, the crossings of branches of the same index
are always possible and classified by $\pi_0^S$ in both cases (for $\Z$ or $\S$ coefficients). The claim follows.
\end{proof}

\begin{example}\label{ex:trefoil}
The sheaves for Legendrian braid closures have been studied in \cite{STZ, STWZ}; their isomorphism classes are the  points of certain spaces of configurations of flags. For example, for the Legendrian trefoil, one can reason as in \cite[Ex. 6.38]{STZ} to find that the $\Z$-sheaves (up to grading shift) are in bijection with the integer solutions to the equation $1+ x + z + xyz=0$.\footnote{This equation is the same as the equation for the augmentation variety of the trefoil, but this identification doesn't follow from the known augmentation/sheaf comparisons (either \cite{NRSSZ} or \cite{GPS3} + \cite{Ekholm-Lekili}) because both \cite{NRSSZ} and \cite{Ekholm-Lekili} require field coefficients.  It would be
good to have a version of these results over $\Z$ (and even better to have a version over $\S$).
}
    These are readily classified:
    $(x, 0,1-x)$, $(-1, 1, z)$, $(x, 1, -1)$,
    $(1, -2, 2)$, $(2, -2, 1)$, $(1, -3, 1)$.
    Thus generating functions, sheaves over $\Z$ and sheaves over $\S$ for the Legendrian trefoil are all in bijection with this (infinite) set.
In terms of generating functions, the integers $x,y,z$ can be understood as an algebraic count of handleslides
that happen before each crossing.
\end{example}

The map $\pi_0\GF_\Lambda(M)^*\to \pi_0\mu\GF_\Lambda(M)^*$ is typically not surjective.
In the $1$-dimensional case, the spaces of microlocal objects $\mu\GF_\Lambda(M)^*$, $\mu\Sh_\Lambda^+(M;\S)^*$ and $\mu \Sh_\Lambda(M;\Z)^*$ are not empty if and only if the first Maslov class $\mu_1(\Lambda)\in H^1(\Lambda;\Z)$ vanishes, and, if so, their $\pi_0$ are all isomorphic
and torsors over $H^0(\Lambda;\Z)\oplus H^1(\Lambda;\{\pm 1\})$.
The $H^0(\Lambda;\Z)$-summand acts by shifting the grading (individually on each component of $\Lambda$) while the $H^1(\Lambda;\{\pm 1\})$-summand acts on microsheaves by tensor with a rank one $\Z$ local system.

For a Legendrian link with generic front projection (only cusps and crossings) and vanishing first Maslov class, we may give a trivialization of the second factor of the torsor, i.e. a map $\pi_0\mu\GF_\Lambda(\R)^*\to H^1(\Lambda;\{\pm 1\})$ equivariant with respect to the action of $H^1(\Lambda;\{\pm 1\})$,
 as follows. For $f\in \mu\GF_\Lambda(\R)^*$ we choose orientations
of the local stable manifold of each fiberwise critical point of $f$ along each branch of the front in such a way
that they are compatible at each left cusp (i.e. the Morse differential is $1$ just after the cusp with respect
to these orientations). At each right cusp the orientations may agree or disagree, so we get a sign $\epsilon_i$,
and the product $\prod_i\epsilon_i$ where $i$ ranges over all right cusps of a given component gives the promised map.
We can check that the map is independent of the choices of orientations.
However the above map $\pi_0\mu\GF_\Lambda(\R)^*\to H^1(\Lambda;\{\pm 1 \})$ is not invariant under
Legendrian isotopy; in particular, a swallowtail move (i.e. Reidemeister I) on one component introduces a new right cusp with sign $-1$.

\begin{example}\label{ex:mumonunknot}
For the Legendrian unknot, with front a flying saucer lying above say an interval $[0,1]$, given orientations of the stable manifolds
as described above, we see that at the right cusp the Morse differential must be $1$, since this Morse
differential is well-defined for all $t\in]0,1[$ and must therefore be constant. So the microlocal monodromy is $1$.
Now if we do a swallowtail move, we get a front with one crossing (of degree $0$) and the classification of $\pi_0\GF_\Lambda(M)^*$
involves only one handleslide $x\in \Z$ which has to satisfy $x=1$ (for the object to extend at the right cusps).
The signs at the right cusps are $1$ and $-1$ so the microlocal monodromy is $-1$.
\end{example}

We illustrate with some  closures of $2$-strand Legendrian braids.
For the hopf link (i.e. $2$ crossings), the classification of objects (with admissible Maslov potential $\mu$) yields
\[\pi_0\GF_{\Lambda,\mu}(M)^*=\{(x,y)\in \Z^2, 1+xy=\pm 1\}/\pm 1\]
where the quotient identifies $(x,y)$ with $(-x,-y)$. The sign at each right cusp is $1+xy$ and therefore the
microlocal monodromy is $(1,1)$ or $(-1,-1)$ (so half of the possible
microlocal monodromies are realized by genuine objects).

For the trefoil (i.e. $3$ crossings), keeping track of Morse differentials after the various handleslides $x,y,z \in \Z$,
we see that
\[\pi_0\GF_{\Lambda}(M)^*=\{(x,y,z)\in \Z^3, x+z+xyz=\pm 1\}/\pm 1\]
and the sign at the right cusps are $x+z+xyz$ and $-(x+z+xyz)$ so the product is $-1$ independently of $(x,y,z)$.
To compare with the previous case and with Example~\ref{ex:trefoil}, observe that $\{(x,y,z)\in \Z^3, x+z+xyz=\pm 1\}/{\pm 1}$ is in canonical bijection with
$\{(x,y,z)\in \Z^3, x+z+xyz=1\}$.

The next example (i.e. $4$ crossings) is a (2,4)-torus link and the classification (with admissible Maslov potential $\mu$) reads

\[\pi_0\GF_{\Lambda,\mu}(M)^*=\{(x,y,z,t)\in \Z^4, 1+xy+tx+tz+xyzt=\pm 1\}/\pm 1.\]
The sign at each right cusp is $1+xy+tx+tz+xyzt$ so again the microlocal monodromy is $(1,1)$ or $(-1,-1)$.

It appears that only half of the possible microlocal monodromies can be realized by a genuine generating function, this is indeed
a known result of Akhmetev-Cencelj-Repovs  and reproved by Pushkar-Tyomkin.\footnote{We thank Misha Tyomkin for pointing this out.}
In our language this reads as follows.
\begin{thm}\cite{akhmetev_algebraic_2005,pushkar_enhanced_2021} \label{micromonodromies of gf}
Let $\Lambda$ be a Legendrian link in $J^1(\R)$ with a generic front, and denote $k\in \N$ the number
of double points of this front. Then the product of the microlocal monodromies of each component of $\Lambda$,
associated to a generating function $f\in \GF_\Lambda(\R)^*$ equals $(-1)^k$.
\end{thm}

The corresponding result for augmentations was established by Leverson  \cite{leverson_augmentations_2016} concerning augmentations of the Chekanov-Eliashberg
DGA of Legendrian links. In this context the microlocal monodromy on each component is analogous to the value
$\epsilon(t_i)$ of an augmentation $\epsilon$ on the formal generator $t_i$ corresponding to
the $i$-th component of $\Lambda$ (here the DGA is defined with coefficients in $\Z[t_1,\dots,t_s]$).
Leverson's result asserts that $\epsilon(t_1)\cdots\epsilon(t_s)=(-1)^s$ for all augmentations $\epsilon$.
The corresponding result for sheaves then follows from \cite{NRSSZ},\footnote{While \cite{NRSSZ} assumes field coefficients, the property that a given $\mathbb{Z}$-sheaf has microlocal monodromy $-1$ can be checked after extension of scalars to $\mathbb{Q}$, or for that matter, after reduction mod any prime greater than 2.}
from which we may also deduce Theorem \ref{micromonodromies of gf}.

\vspace{2mm}
We turn now to consider cases where Theorem \ref{thm: Z to S} (and Theorem \ref{strong thm b}) do not apply.

\begin{example}\label{ex:hopflink}
Consider the Hopf link whose front is a $2$-copy of a flying saucer and a Maslov potential $\mu$
which is not admissible, namely it gives Maslov potential $j$ and $j+1$ to the top copy, $i$ and $i+1$ to
the bottom copy and $j>i+1$.
The crossings (and the corresponding Reeb chords after applying Ng's resolution)
therefore have degree $k=j-i-1>0$ and $-k=i+1-j<0$. We claim that
\[\pi_0\GF_{\Lambda,\mu}(M)^*\simeq \pi_0\Sh_{\Lambda,\mu}^+(M;\S)^*\simeq \pi_k^S.\]
Indeed when trying to extend a generating function (or a sheaf over $\S$) from left to right, at the first
crossing there is no obstruction (or rather the obstruction in $\pi_{k-1}^S$ vanishes) since the critical points
are just born in different places, and the spaces of crossings has $\pi_0$ equal to $\pi_k^S$. At the second crossing
the obstruction is valued in the trivial group and $\pi_0$ of the space of crossings is a torsor over
the trivial group (since $\pi_{-k}^S=\pi_{-k-1}^S=0$). The choice that has been made at the first crossing does not affect the Morse
differentials (it is invisible to homology with $\Z$-coefficients) and therefore the extension is possible at
the right cusps (the Morse differential is $1$). On the other hand we have
\[\pi_0 \Sh_\Lambda^+(M;\Z)=0\]
since the choice in $\pi_k^S$ with $k>0$ does not exist for sheaves over $\Z$.
We learn that the map
\[\pi_0\Sh_\Lambda^+(M;\S)\to \pi_0\Sh_\Lambda^+(M;\Z)\]
is not injective in general and that already for the Hopf link the full classification of generating functions
involves all stable homotopy groups of spheres.
\end{example}

For any Legendrian link $\Lambda \subset J^1(\R)$,  it follows from \cite[Thm. C]{Fuchs-Rutherford} (there attributed to \cite{chekanov-pushkar-combinatorics}) that if a generating sheaf $F \in \Sh^+_\Lambda(\R,k)^*$ exists for some field $k$, then $\Lambda$ admits a generating function $f$.  In general, said result says nothing about the relationship between $F$ and $\s(f)$, except when $k=\Z/2$, where, by examining their proof, we find:

\begin{prop}\label{prop:mod2lift}
For any Legendrian link $\Lambda \subset J^1(\R)$, the maps
\[\pi_0 \GF_\Lambda(\R)^* \to \pi_0 \Sh^+_\Lambda(\R;\Z/2)\]
and
\[\pi_0 \Sh^+_\Lambda(\R;\Z)^* \to \pi_0\Sh^+_{\Lambda}(\R;\Z/2)\]
are surjective.
\end{prop}

\begin{proof}
Because of the factorization
\[\pi_0 \GF_\Lambda(\R)^* \to \pi_0 \Sh^+_\Lambda(\R;\Z) \to \pi_0 \Sh^+_\Lambda(\R;\Z/2)\]
it is enough to prove that the composition is surjective.

We now recall the construction from \cite{Fuchs-Rutherford}.
Assume without loss of generality that the front of $\Lambda$ is generic. Let $F\in \Sh^+(\Lambda;\Z/2)$.
We try and construct a generating function $f$ lifting $F$ from left to right, starting with the function
$f:\R\times \R\times \R^n\to \R$ given by $f(x;u,v)=u$ for $x<<0$. We will construct $f$ together with a gradient vector
field $X$ such that, for $x$ away from a small neighborhood of the bifurcation values $x_i$ (where $\Lambda$ has a crossing or a birth/death),
there is a partition of $\crit(f_x)$ in pairs which are joined by a unique $X$-trajectory. Let us call
$X$ a \emph{simple} gradient vector field.

We now explain how to construct this lifting $(f,X)$ past each bifurcation.
\begin{itemize}
\item At each birth, we simply introduce a new pair of critical points in a small box ($n$ is chosen large enough)
in such a way that there are no $X$-trajectories between these critical points and the other ones.

\item At each death, the sheaf $F$ tells us that the number of trajectories between the critical
points about to annihilate is odd, but since it must be $0$ or $1$ by the property of $X$, it is $1$.
Hence we can cancel the critical points.

\item At each crossing of different Maslov potential $i\neq j$ ($j$ being above $i$ before the crossing).
If $j \neq i+1$, there is no $X$-trajectory between the critical points and therefore we can simply extend $f$
without changing the vector field $X$. If $j=i+1$, the sheaf $F$ tells us that the number of $X$-trajectories
between these critical points is even, but it is $0$ or $1$ by construction, so it is $0$
and the crossing is possible (again with fixed $X$). This will still lift the sheaf $F$ since the extension
of the sheaf past the crossing is unique anyway.

\item  At each crossing of equal Maslov potential, there are two possible extensions of the sheaf
past the crossing differing by a ``$\Z/2$-valued handleslide''. The given $F$ is one of these two extensions
and is determined by whether or not the pairing
of critical values is preserved or not past the crossing. In either case, we can
realize it by a generating function (essentially doing one handleslide before the crossing
and one handleslide after) to obtain a generating function $f$ lifting $F$ (see \cite{Fuchs-Rutherford})
and a simple gradient field $X$.
\end{itemize}
\end{proof}

\begin{remark}
The original formulation in \cite{Fuchs-Rutherford} asserted that one could lift a ``graded normal ruling'' rather than a $\Z/2\Z$-sheaf.  A graded normal ruling is a fixed-point-free involution of the smooth locus of the front of $\Lambda \subset J^1(\R)$, which preserves the projection to $\R$, and satisfying certain explicit conditions at the cusps and crossings.  Given a generating sheaf $F \in \Sh^+_\Lambda(\R,k)$ for $k$ a field, one constructs a graded normal ruling as follows: for $x \in
R$, consider the restriction $F_x \in \Sh^+(\mathrm{point}) \hookrightarrow \Sh_{\tau \ge 0}(\R)$.  One checks that necessarily $F_x = \bigoplus k_{[a,b)}$, where the collection of $(x,a), (x,b)$ are the coordinates of points over $x$ in the front of $\Lambda$.  The graded normal ruling for the generating sheaves associated to the construction in Proposition \ref{prop:mod2lift} are given by the ``partition of $\crit(f_x)$ in pairs which are joined by a unique $X$-trajectory''.

Sheaves over $\Z/2$ carry more information than the induced graded normal ruling.
For example, for the trefoil $\Lambda \subset J^1(\R)$, there are, up to shift, $5$ elements of $\pi_0 \Sh^+_\Lambda(\R; \Z/2)^*$ \cite[Example 6.38]{STZ}.  However $\Lambda$ only has $3$ graded normal rulings (up to shift).
\end{remark}

\begin{remark}
The map
$\pi_0\Sh^+_\Lambda(\R;\Z)^* \to \pi_0 \Sh^+_\Lambda(\R;\Z/5)^*$
is typically not surjective. One reason is that the microlocal monodromy has to be $\pm 1$
for a $\Z$-sheaf while it can be any non-zero element of $\Z/5$ for a $\Z/5$-sheaf.
For example, there is a choice of Maslov potential for the Hopf link $\Lambda \subset J^1(\R)$ such that, for any ring $k$, the elements of $F\in \Sh^+_\Lambda(\R;k)^*$ are in bijection with pairs $(x,y)\in k^2$ satisfying $1+xy \neq 0$, and have microlocal monodromy $1+xy$.  For $k = \Z/5$, one can take $x=1$ and $y=2$; then the corresponding object has microlocal monodromy $3$, hence  cannot be lifted to a $\Z$-sheaf.
\end{remark}

The generating functions and sheaves over $\Z$ produced in the proof of Proposition~\ref{prop:mod2lift}.
are very particular.  We remain ignorant as regards the following:

\begin{question}
For a Legendrian link $\Lambda \subset J^1(\R)$, is the map
\[\pi_0\GF_\Lambda(\R)^*\to \pi_0\Sh^+_\Lambda(\R;\Z)^*\]
always surjective ?
\end{question}

\appendix

\section{Fiberwise tame functions} \label{app:fibinfty}

Here we collect some properties of fiberwise tame functions.
Recall the following standard definition.
\begin{dfn}
A map of manifolds $f:X\to Y$ is a \emph{fibration} if for each $y\in Y$ there exists an open neighborhood $U$ of $y$
and a diffeomorphism $\varphi:f^{-1}(U) \to U\times f^{-1}(y)$ such that $f=\pr_1\circ\varphi$.
\end{dfn}

The image of a fibration is open and closed, but we do not require it to be all of $Y$. 

\begin{dfn}
A function $f: E \to \R$ is {\em fiberwise tame} with respect to a fibration $p:E\to M$ if there is a closed set $K\subset E$ and a smooth function $a:M\to (0,+\infty)$
such that the projection $K\to M$ is proper and the maps
\[\{e\in E, |f(e|\geq a(p(e))\} \to \{(x,z)\in M\times \R, |z|\geq a(x)\},\]
and
\[\{e\in E, |f(e|\leq a(p(e))\} \setminus K \to \{(x,z)\in M\times \R, |z|\leq a(x)\}\]
are fibrations.

We denote $\Sigma_{(p,f)}$ the set of critical points of $(p,f):E\to M\times \R$ or, equivalently, the set of fiberwise critical points of $f$.
\end{dfn}

\begin{rem}
The definition appears (at least with $M$ a single point) in \cite{Eliashberg-Gromov}
under the name "fibration at infinity".
\end{rem}

\begin{rem}
Some properties:
\begin{itemize}
\item The notion of fiberwise tameness is invariant by left composition with fibered diffeomorphism of $M\times \R$ and by right composition with fibered diffeomorphisms of $E$.
\item The set $\Sigma_{(p,f)}$ maps properly to $M$.
\item If $p:E\to M$ is proper any function $f:E\to \R$ is fiberwise tame with respect to $p$.
\item The map $(p,f):E\to M\times \R$ is a fibration if and only if $f$ is fiberwise tame with respect to $p$ and $\Sigma_{(p,f)}$ is empty.
\item The subset $K$ can be taken to be a smooth codimension $0$ submanifold with boundary fibered over $M$.
\end{itemize}
\end{rem}

The notion of fiberwise tameness is clearly invariant under pullback:
\begin{lem}\label{lem:tamepullback}
If $M'\to M$ is any map and $f:E\to \R$ a fiberwise tame functions with respect to a fibration $p:E\to M$, then $f$ restricted to
$E\times_M M'$ is also fiberwise tame with respect to the fibration $E\times_M M'\to M'$.
\end{lem}

Concerning pushforward, we have the following lemma.
\begin{lem}\label{lem:tamepushforward}
Let $q:M\to M'$ be a fibration, $f:E\to \R$ be a fiberwise tame function with respect a fibration $p:E\to M$ and $p'=q\circ p$.
If $\Sigma_{(p,f)} \to M'\times \R$ and $\Sigma_{(p',f)}\to M'$ are proper then $f$ is also fiberwise tame with respect to the fibration $p'$.
\end{lem}

\begin{proof}
We prove it in the case where $M'$ is a single point. We have to find a compact codimension $0$ submanifold $K\subset M$ and $a>0$ such that $\{|f|\leq a\}\setminus K \to [-a,a]$ and $\{|f|\geq a\}\to \R\setminus(-a,a)$ are fibrations.

The fiberwise tameness provides $b
:M\to (0,+\infty)$ and a codimension $0$ subbundle $A\subset E$ of $E\to M$ such that $\{|f|\leq b\}\setminus A\to \{|z|\leq b\}\subset M\times \R$ and $\{|f|\geq b\}\to \{|z|\geq b\}\subset M\times \R$ are fibrations.

Since $\Sigma_{(p',f)}$ is proper over $M'$ we may  find a  codimension $0$ submanifold $N\subset M$ which properly maps to $M'$ and such that $p(\Sigma_f)\subset \mathrm{int}(N)$.
Up to increasing $b$ on $N$ we may assume that it is equal to $a\circ q$ on $N$ for some function $a:M'\to (0,+\infty)$ and $b\geq a\circ q$ on $M$. We set $K=A\cap p^{-1}(N)$ and now check that $f$ satisfies the tameness condition with $a$ and $K$.

We already know that $\{|f|\leq b\} \setminus A \to \{(m,z)\in M\times \R, |z|\leq a(m)\}$ is a fibration, we need to take care of the region $A\setminus K$. By assumption $\Sigma_{(p,f)}\to M'\times\R$ is proper, so $\Sigma_{(p,f)}\cap\{|f|\leq a\}$ is contained in $p^{-1}(N')$ for some codimension $0$ submanifold $N$' containing $N$ in its interior and properly mapping to $M'$. Above the region $M\setminus N'$, the function $f$ has no fiberwise critical points on the subbundle $A$, therefore $A\cap p^{-1}(M\setminus N')$ is fibered over $\{(m',z)\in M'\times \R, |z|\leq a(m')\}$. Over the region $N'\setminus N$, 
the function $f$ may have $p$-fiberwise critical points on the subbundle $A$ but by assumption it has no $p'$-fiberwise critical points, so by properness of $A\cap p^{-1}(N'\setminus \mathrm{int}(N))$ over $M'$, the latter region is also fibered over its image. It follows that $\{|f|\leq a\}\setminus K \to \{(m',z)\in M'\times \R, |z|\leq a(m')\}$ is a fibration.

We now pick $c>a$ and show that $\{a\leq f\leq c\}\to \{(m',z), a(m')\leq z\leq c(m')\}$ is a fibration. We already know that $\{b\geq f\leq c\}$ is fibered over $\{b\leq z\leq c\}\subset M\times \R$. Since $\Sigma_{p,f}\to M'\times \R$ is proper, we may find a codimension $0$ submanifold $N'$ of $M$ which properly maps to $M'$ and such that $\Sigma_{p,f}\cap\{a\leq f \leq c\}$ is contained in $p^{-1}(N')$. Above $M\setminus N'$ the function $f$ has no $p$-fiberwise critical points on the subbundle $A\cap \{a\leq f\leq b\}$ so it is a fibration there. Above the region $N'$ the function $f$ has no $p'$-fiberwise critical points on $A\cap\{a\leq f \leq c\}$ so it is also a fibration there. It follows that $\{a\leq f\leq c\}\to \{(m',z), a(m')\leq z\leq c(m')\}$.
\end{proof}

The preceding lemma can be typically used with the map $q$ from $M$ to a point to guarantee that a fiberwise tame function is also tame as a function of all variables.

Another useful fact is the stability of fiberwise tame functions under addition:
\begin{lem}\label{lem:tamesum}
Let $p:E\to M$ and $p':E'\to M'$ be fibrations and $f:E\to \R$, $f':E'\to \R$ be fiberwise tame functions (with respect to $p$ and $p'$ respectively).
Then the sum $f+f':E\times E'\to \R$ is also fiberwise tame with respect to the projection $p\times p' : E\times E'\to M\times M'$.
\end{lem}

\begin{proof}
By assumption we have, $a, a'>0$ and proper subbundles $A$ and $A'$ of $E$ and $E'$ respectively
such that $\{|f|\leq a\}\setminus A\to \{|z|\leq a\}$,  $\{|f'|\leq a'\}\setminus A'\to \{|z|\leq a'\}$, $\{|f|\geq a\} \to \{|z|\geq a\}$ and $\{|f'|\geq a\} \to \{|z|\geq a'\}$ are fibrations.
We pick complete fiberwise gradient vector files $X$ and $X'$ for $f$ and $f'$ respectively which
are adapted to $A$ and $A$', namely $X$ points inward transversely along $A\cap\{z=-a\}$, outward along $A\cap\{z=a\}$ and is tangent to $\del A$ in the remaining parts of $\del A$.
All fiberwise critical points of $f+f'$ are contained in $A\times A'$, and we may therefore saturate $A\times A'$ along the flow of $X+ X'$ to obtain a subbundle $B\subset \{|f|\leq a+a'\}\subset E\times E'$
with boundary decomposed in three parts lying respectively in $\{f+f'=a+a'\}$, $\{f+f'=-a-a'\}$
and the remaining part fibered over $[-a-a',a+a']$. Using the flow of $X+X'$, we can then check that $f+f'$ satisfies the fiberwise tameness condition with $B$ and $a+a'$.
\end{proof}

\begin{cor}\label{cor:tamefiberedsum}
Let $p:E\to M$ and $p':E'\to M$ be fibrations and $f:E\to \R$, $f':E'\to \R$ be fiberwise tame functions (with respect to $p$ and $p'$ respectively).
Then the fibered sum $f+f':E\times_M E'\to \R$ is also fiberwise tame with respect to the induced projection $E\times_M E' \to M$.
\end{cor}
\begin{proof}
It follows directly from Lemma~\ref{lem:tamesum} and Lemma~\ref{lem:tamepullback} using pullback along the diagonal map $M\to M\times M$.
\end{proof}

Finally we record the following result which says that a $1$-parameter deformation of a (fiberwise) tame function is equivalent to a compactly supported deformation.

\begin{prop}\label{prop:tamedeformation}
Let $p:E\to M$ be a fibration and $f:E\times[0,1]\to \R$ a fiberwise tame function with respect to $p\times \id:E\times [0,1]\to M\times[0,1]$,
then there exists a closed subset $A\subset E$ such that $A\to M$ is proper and a diffeomorphism $\varphi:E\times [0,1]\to E\times [0,1]$
fibered over $M\times [0,1]$ such that $f\circ \varphi(x,t)=f(x,0)$ for $x\not\in A$.
\end{prop}
\begin{proof}
We write $f(e,t)=f_t(e)$ for $(e,t) \in E\times [0,1]$.
Pick a subbundle $A$ of $E$ which properly maps to $M\times[0,1]$ and $a:M\times [0,1]\to (0,\infty)$ as in the definition of fiberwise tameness. We may assume that $a$ is independent of $t$ by compactness of $[0,1]$. Using the fibration properties from the definition of fiberwise tameness we can find a vector field $Y_t=(X_t,1)$ on $E\times [0,1]$ which can be integrated up to time $1$ and such that $dp(X_t)=0$ everywhere and $df_t(X_t)+ \frac{df_t}{dt}=0$ away from $A$. The time-dependent vector field generates an isotopy $(\varphi_t)_{t\in [0,1]}$ defined by $\varphi_0=\id$ and $\frac{d \varphi_t}{dt}=X_t\circ\varphi_t$. Integrating this equation we find $f_t\circ \varphi_t = f_0$ away from some closed set proper over $M$, so the fibered diffeomorphism of $E\times [0,1]$ defined by $(x,t)\mapsto (\varphi_t(x),t)$ has the desired properties.
\end{proof}

\bibliographystyle{plain}
\bibliography{biblio}

\end{document}